\documentclass[12pt]{amsart}
\usepackage{graphicx}
\usepackage{overpic}

\usepackage{amsmath}
\usepackage{amssymb} 
\usepackage{amsthm}
\usepackage{amscd}
\usepackage{cite}
\usepackage{latexsym}
\usepackage{units}

\newcommand{\R}{\mathbb{R}}
\newcommand{\Z}{\mathbb{Z}}

\newcommand{\bdry}{\partial}

\newcommand{\lk}{\textrm{lk}}

\renewcommand{\d}{\partial}
\newcommand{\dd}[1]{\frac{\partial}{\partial{#1}}}
\newcommand{\Int}{\operatorname{Int}}

\newcommand{\cnm}{C_n[M]}
\newcommand{\cnmo}{C_n(M)}

\newcommand{\pij}{\pi_{ij}}
\newcommand{\sijk}{s_{ijk}}

\newcommand{\onto}{\rightarrow}

\def\co{\colon\thinspace}

\newtheorem{theorem}{Theorem}[section]
\newtheorem*{theorem53}{Theorem~\ref{general-change}}
\newtheorem*{thm-b1}{Theorem~\ref{alexander basis existence}}
\newtheorem*{thm-b3}{Theorem~\ref{alexander basis transformations}}

\newtheorem{lemma}[theorem]{Lemma}
\newtheorem{proposition}[theorem]{Proposition}
\newtheorem{corollary}[theorem]{Corollary}
\newtheorem{question}[theorem]{Question}

\theoremstyle{definition}
\newtheorem{definition}[theorem]{Definition}

\newcommand \ctwoo {C_2[\Omega]}
\newcommand \ao {A_2[\Omega]}
\newcommand \E{\mathfrak{E}}
\newcommand\bdy{\partial}

\newcommand\ddx[1]{\frac{\d}{\d x_{#1}}}
\newcommand\ddy[1]{\frac{\d}{\d y_{#1}}}
\newcommand\ddu[1]{\frac{\d}{\d u_{#1}}}
\newcommand\ds{\displaystyle}

\newcommand{\Embed}{\mathfrak{Embed}(\Omega \hookrightarrow \R^{2k+1})}
\newcommand{\Emb}{\mathfrak{E}}
\newcommand\ocomp{\bar\Omega}

\newcommand{\curl}[1]{\nabla\times {#1}}
\newcommand{\diver}[1]{\nabla\cdot {#1}}
\newcommand\cross{\times}
\newcommand\norm[1]{\left| #1 \right|}
\newcommand\Hel{\operatorname{H}}
\newcommand\BS{\operatorname{BS}}
\newcommand\Vol{\operatorname{vol}}
\newcommand\dVol{\operatorname{dvol}}
\newcommand\Lk{\operatorname{Lk}}

\newcommand\HP{\operatorname{HP}}
\newcommand\Torelli{\operatorname{Torelli}}

\newcommand\Flux{\operatorname{Flux}}
\newcommand\tbeta{\tilde{\beta}}
\newcommand\tgamma{\tilde{\gamma}}
\newcommand\tphi{\tilde{\phi}}
\newcommand\fia{\left( f^{-1} \right)^*\alpha}

\begin{document}

\title[new cohomological formula for helicity]{A new cohomological formula for helicity in $\R^{2k+1}$ reveals the effect of a diffeomorphism on helicity}
\author{Jason Cantarella}
\address{Department of Mathematics, University of Georgia, Athens, GA 30602}
\email{cantarel@math.uga.edu}
\author{Jason Parsley}
\address{Department of Mathematics, Wake Forest University, Winston-Salem, NC 27109}
\email{parslerj@wfu.edu}
\subjclass{Primary: 57R25; Secondary: 82D10,82D15}
\keywords{helicity of vector fields, configuration spaces, Bott-Taubes integration} 
 
\begin{abstract}
The helicity of a vector field is a measure of the average linking of pairs of integral curves of the field.  Computed by a six-dimensional integral, it is widely useful in the physics of fluids.  For a divergence-free field tangent to the boundary of a domain in 3-space, helicity is known to be invariant under volume-preserving diffeomorphisms of the domain that are homotopic to the identity.  We give a new construction of helicity for closed $(k+1)$-forms on a domain in $(2k+1)$-space that vanish when pulled back to the boundary of the domain.  Our construction expresses helicity in terms of a cohomology class represented by the form when pulled back to the compactified configuration space of pairs of points in the domain.  We show that our definition is equivalent to the standard one.  We use our construction to give a new formula for computing helicity by a four-dimensional integral.  We provide a Biot-Savart operator that computes a primitive for such forms; utilizing it, we obtain another formula for helicity.  As a main result, we find a general formula for how much the value of helicity changes when the form is pushed forward by a diffeomorphism of the domain; it relies upon understanding the effect of the diffeomorphism on the homology of the domain and the de Rham cohomology class represented by the form.  Our formula allows us to classify the helicity-preserving diffeomorphisms on a given domain, finding new helicity-preserving diffeomorphisms on the two-holed solid torus and proving that there are no new helicity-preserving diffeomorphisms on the standard solid torus.  We conclude by defining helicities for forms on submanifolds of Euclidean space.  In addition, we provide a detailed exposition of some standard `folk' theorems about the cohomology of the boundary of domains in $\R^{2k+1}$.
\end{abstract}
\maketitle
\section{Introduction}

The linking number of a pair of closed curves $a$ and $b$ in $\R^3$ is a topological measure of their entanglement. We can define the linking number as the degree of the Gauss map $g \co S^1 \cross S^1 \rightarrow S^2$ given by $g(\theta,\phi) =\left( a(\theta) - b(\phi) \right) / \norm{a(\theta) - b(\phi)}$. This degree can be written combinatorially, by counting signed crossings of $a$ and $b$, but we can also write this degree as an integral by pulling back the area form on $S^2$ via the Gauss map and integrating over the torus $S^1 \cross S^1$. This ``Gauss integral formula'' for linking number yields 
\begin{equation*}
\Lk(a,b) = \frac{1}{\Vol(S^2)} \int a'(\theta) \cross b'(\phi) \cdot \frac{a(\theta) - b(\phi)}{\norm{a(\theta) - b(\phi)}^3} d\theta \, d\phi.
\end{equation*}
The linking number is a knot invariant, so it is invariant under any ambient isotopy of $\R^3$ carrying the curves to new curves $\tilde{a}$ and $\tilde{b}$. 

Given a divergence-free vector field $V$ on a domain $\Omega \subset \R^3$ with smooth boundary, we can define an analogous integral invariant known as \emph{helicity}.  The six-dimensional helicity integral, which measures the average linking number of pairs of integral curves of the field \cite{MR891881}, is given by:
\begin{equation} 
\label{mhelicity}
\Hel(V) = \frac{1}{\Vol(S^2)} \int_{\Omega \times \Omega} {V(x) \times V(y) \cdot \frac{x-y}{|x-y|^3} \; \dVol_x \, \dVol_y }
\end{equation}
Just as the linking number of a pair of curves is a knot invariant, we might expect the helicity of a vector field to be a diffeomorphism invariant. This is not always true, as we will demonstrate below, but it is true in enough cases to make helicity an important quantity in fluid dynamics and plasma physics~\cite{MR819398}. 

The helicity invariant for vector fields was used in plasma physics as early as 1958 by L.\ Woltjer~\cite{MR0096542}. Woltjer showed that helicity was an invariant of the equations of ideal magnetohydrodynamics for an isolated system, and as such it was immediately useful in the study of astrophysical plasmas.  J.J.\ Moreau in 1961~\cite{MR0128195} 
first used the invariant to study fluid dynamics.  In an influential 1969 paper~\cite{mof1}, Keith Moffatt proved that helicity is an invariant of the equations of ideal fluid flow, even in the presence of an external force on the fluid. 

The invariance of helicity has been reproved in various physical contexts ever since.  For instance, Peradzynski showed that helicity was invariant under the equations of motion for superfluid helium~\cite{peradzynski}. The same invariant was associated to foliations by Godbillon and Vey in 1970~\cite{MR0283816}, by defining the foliation as the kernel of a 1-form and measuring the helicity of the form. In 1973, V.I.\ Arnol'd defined helicity for 2-forms in a 3-manifold~\cite{MR891881} (the paper was published in English translation in 1986). His may be the first proof of the invariance of helicity under arbitrary volume-preserving diffeomorphisms (on a simply-connected domain)\footnote{A corresponding theorem for the Godbillon-Vey invariant appears in a paper of G.\ Raby from 1988~\cite{MR949006}.}.  

The most general invariance theorem for helicity known is:
\begin{theorem}[Invariance of helicity theorem, \cite{MR1612569, MR1770976}]
\label{classicalinvariance}
The helicity of a divergence-free vector field $V$ on a domain $\Omega \subset \R^3$ is invariant under any volume-preserving diffeomorphism of $\Omega$ which is homotopic to the identity.  If $\Omega$ is simply connected,  
then helicity is invariant under any volume-preserving diffeomorphism of $\Omega$.  

If $V$ is a null-homologous vector field (meaning that its dual 2-form is exact) on a compact manifold $M^3$ without boundary, then helicity is invariant under any volume-preserving diffeomorphism of $M$.
\end{theorem}

Also, if $V$ on $\Omega$ is fluxless (cf. section~\ref{sec:fluxless}) on a domain in $\R^3$ with boundary, then its helicity is invariant under any volume-preserving diffeomorphism~\cite{MR1770976}.  These invariance results leave open some natural questions: are these all of the helicity-preserving diffeomorphisms?  If not, can we classify the diffeomorphisms that do preserve helicity?  What is the effect of an arbitrary diffeomorphism on helicity?

Figure~\ref{tori} depicts a diffeomorphism which does not preserve helicity.  Here, the domain is a solid torus, and the vector field following its longitudes is divergence-free and null-homologous\footnote{For domains in $\R^3$ with boundary, it is the fluxless condition, and not the null-homologous condition, which guarantees that helicity remains unchanged under all volume-preserving diffeomorphisms.  We will give a homological interpretation of this fact in section~\ref{sec:invariance}.}.
  Applying a Dehn twist will preserve the volume form but changes the helicity of the field, which we will calculate by Theorem~\ref{torus-h}. 
 
\begin{figure}[ht]
\begin{overpic}{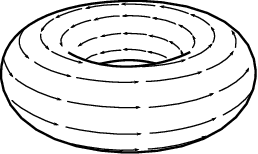}
\end{overpic}
\hspace{0.5in}
\begin{overpic}{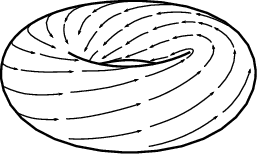}
\end{overpic}
\caption{The figure shows the effect of a diffeomorphism $f$ which applies a Dehn twist to the solid torus $\Omega$ on a vector field $V$ dual to a 2-form $\alpha$. If the radius of the tube is $R$, the helicity of the left hand field is $0$ and the helicity of the right hand field is equal to the square of the flux of $V$ across a spanning surface for the tube: $\pi^2 R^4$}
\label{tori}
\end{figure}

To answer the questions above, we notice that in the theory developed so far, there exists an asymmetry between linking number and helicity -- while there are several useful ways to obtain linking number, including a ``purely homological'' expression as the degree of a map and a combinational expression as the sum of signed crossing numbers as well as an integral expression, so far the helicity has only been expressed as an integral.  In this paper we try to restore the balance between linking number and helicity by providing a purely cohomological definition for the helicity of $(k+1)$-forms on domains $\Omega$ in $\R^{2k+1}$ (Definition~\ref{fthelicity}).  We work with forms $\omega$ that are closed and satisfy the following definition:

\begin{definition}
A smooth $p$-form $\alpha$ defined on the domain $\Omega$ is a \emph{Dirichlet form} if $\alpha$ vanishes when restricted to the boundary, i.e., if $V_1, \ldots, V_p$ are all tangent to $\bdy \Omega$, then $\alpha(V_1, \ldots, V_p)=0$. 
\end{definition}  

For domains in $\R^3$, closed Dirichlet forms are dual to vector fields that are divergence-free and tangent to the boundary.  In Appendix~\ref{hodge}, we examine decompositions of differential forms; in particular we characterize the set of closed Dirichlet forms.  Proposition~\ref{alphaisexact} guarantees that every closed Dirichlet form is exact.

Arnol'd defined helicity as the integral of the wedge product of an exact form with a primitive for that form~\cite{MR1612569}.  But there are many primitives for a given exact form on a general domain -- so it is clear that the choice of primitive must be important to the definition. The helicity integral implicitly solves this problem by constructing a particular inverse curl for the given vector field by integration. Khesin and Chekanov generalized this approach to forms by defining an primitive for a form on $\Omega$ by integrating the wedge product of the pullback of the form to $\Omega \cross \Omega$ and a singular ``linking form'' over the fiber in the bundle $\Omega \cross \Omega \rightarrow \Omega$~\cite{MR1028280,MR1612569}. 

Our paper is divided into two parts. In the first part of our paper, we redevelop this standard theory of helicity using a new idea: instead of defining a singular linking form on $\Omega \cross \Omega$, we use a nonsingular form on the compactified configuration space $C_2[\Omega]$ of pairs of distinct points on $\Omega$. This approach allows us to give simpler and clearer proofs of the standard results on helicity for forms which better expose the underlying topology of the problem.

Here is a summary of our construction for helicity in a simple special case.  Suppose that on a ball $\Omega \subset \R^3$ we consider the helicity of a divergence-free vector field $V$ that is tangent to $\bdy\Omega$. If we associate a 2-form $\alpha$ to $V$ by pairing $V$ with the volume form in $\R^3$, we will prove that the helicity can be expressed as an integral over the 6-dimensional compactified configuration space $C_2[\Omega]$ of disjoint pairs of points in $\Omega$. 

Let us understand the topology of this configuration space.  We note that $C_2[\Omega]$ is homeomorphic to $\Omega \cross \left( \Omega - B_r(x) \right)$, where $B_r(x)$ is a small neighborhood of a point in $\Omega$. Since $\Omega$ is a ball, $\Omega - B_r(x) \simeq S^2 \cross I$ and this space is $D^3 \cross S^2 \cross I \simeq D^4 \cross S^2$. The 2-form $\alpha$ can be pulled back to a pair of 2-forms $\alpha_x$ and $\alpha_y$ on $C_2[\Omega]$ under the projections of $\{x,y\}$ to $x$ and $y$. We will show that $\alpha_x \wedge \alpha_y$ is a closed Dirichlet 4-form on $C_2[\Omega]$. Hence $\alpha_x \wedge \alpha_y$ will represent a class $h[g]$ in the 1-dimensional relative de Rham cohomology group $H^4(D^4 \cross S^2,\bdy (D^4 \cross S^2))$ where $g$ is a generator. We show that if $[g]$ is the Poincar\'{e} dual of the standard area form on $S^2$, then $h$ is the helicity of $V$ divided by the square of the volume of $\Omega$. This gives a cohomological definition of helicity.

This definition has several attractive consequences.  First, our configuration space approach extends to the standard setting for higher dimensional helicity: $(k+1)$-forms on $(2k+1)$-dimensional domains \cite[Section 5]{MR2078570} 
.  We recover the expected result that for $(4n + 1)$-dimensional domains (that is, for even values of~$k$) helicity can extend only to a function that is identically zero, but for $(4n+3)$-dimensional domains, helicity is a nontrivial invariant for differential forms.  We will see immediately that for forms, the question of whether a given diffeomorphism is volume-preserving has no bearing on whether the diffeomorphism preserves helicity.  We are then able to give a quick proof of some of the standard invariance results for helicity in Proposition~\ref{homotopic-invariance}. Some new constructions are immediately suggested by our definition: a new, ``combinatorial'' integral for helicity appears in Proposition~\ref{prop:combinatorial}.   
We complete our redevelopment of the standard theory by proving that our helicity integral can be written as integration against an appropriately chosen primitive in Proposition~\ref{prop:potential}. 

We then begin the second part of paper, in which we solve the problem of computing the effect of an arbitrary diffeomorphism $f\co \Omega \rightarrow \Omega'$ on the helicity of a closed Dirichlet $(k+1)$-form $\alpha$. The theorems in this section are all new.

It is a standard fact (see Appendix B, Theorems~\ref{alexander basis existence} and~\ref{alexander basis transformations}) that the $k$-th homology of $\bdy\Omega$ splits into two subspaces generated by cycles $s_1, \dots, s_n$ which bound relative $(k+1)$-cycles outside $\Omega$ and Poincar\'{e} dual cycles $t_1, \dots, t_n$ which bound relative $(k+1)$-cycles $\tau_1, \dots, \tau_n$ inside~$\Omega$. With respect to a corresponding basis $\langle s'_1, \dots, s'_n, t'_1, \dots, t'_n \rangle$ for the $k$-th homology of $\bdy\Omega'$, we can write the linear map $\partial f_*\co H_k(\bdy\Omega) \rightarrow H_k(\bdy\Omega')$ as a block matrix
\begin{equation*}
\bdy f_* = \left[ 
	\begin{array}{c|c}
	I & 0 \\
	\hline 
	(c_{ij}) & I
	\end{array}
	\right],
\end{equation*}
where the $c_{ij}$ form a symmetric matrix when $k$ is odd and a skew-symmetric matrix when $k$ is even. We then have

\begin{theorem53}
Let $\Omega^{2k+1}$ be a subdomain of $\R^{2k+1}$, and $f\co\Omega \rightarrow \Omega'$ be an orientation-preserving diffeomorphism.  Consider a closed Dirichlet $(k+1)$-form $\alpha$.  The change in the helicity of $\alpha$ under $f$ is
\begin{equation}
\Hel\left(\alpha' \right) - \Hel(\alpha) = \sum_{i,j} c_{ij} \cdot \Flux(\alpha, \tau_i) \Flux(\alpha, \tau_j) \tag{\ref{eq:general-change}}
\end{equation}
where the constants $c_{ij}$ arise from the homology isomorphism induced by $f$ on $H_k(\bdy\Omega)$ as above.  The $(2m+2)$-form $\alpha'$ is the `push-forward' of $\alpha$ under $f$; more precisely, $\alpha'=\fia$ is the pullback of $\alpha$ under the inverse diffeomorphism. \end{theorem53}

Note that for $k$ even (i.e., subdomains of $\R^{4m+1}$) the matrix $\left( c_{ij} \right)$ is skew-symmetric. So Theorem~\ref{general-change} implies that helicity does not change under any diffeomorphism of $\Omega$. This confirms our previous calculation that helicity is always zero in these dimensions.

The simplest example of this theorem is attractive and easy to understand: a diffeomorphism of a solid torus in $\R^3$ isotopic to $j$ Dehn twists changes the helicity of a $2$-form on the torus by $j$ times the square of the integral of the form over a spanning disk, as in Figure~\ref{tori}.

In general, this allows us (for $k$ odd) to classify the helicity-preserving diffeomorphisms from $\Omega$ to $\Omega$ as those maps for which the $c_{ij}$ are all zero. If a diffeomorphism acts trivially on the homology of $\Omega$, it is in this class if and only if it acts trivially on the homology of $\bdy\Omega$ (Corollary~\ref{classification}).

We finish our paper with some discussion of directions for future research, including defining the $(k,n,m)$-helicity of $(k+1)$-forms on $n$-dimensional submanifolds of $\R^m$, an application of our results to computing ``cross-helicities'' of vector fields in two disjoint domains, and some thoughts on defining generalized helicities in a way inspired by the construction of the finite-type invariants for knots.

\section{Defining helicity in terms of cohomology}

Helicity is motivated by the classical linking number between $k$ and $l$ cycles in $\R^{k+l+1}$. If $k=l$, observe that this linking number is only defined in odd-dimensional ambient spaces, which is why the classical linking number and helicity are defined in $\R^3$. In general, one could attempt to define the helicity of a tuple of $k$ vector fields in $\R^{2k+1}$, applying a version of the Gauss linking integral. We find it more natural to write such a tuple as dual to a single $(k+1)$-form, which can be envisioned as the form constructed by contracting the $k$-tuple of vector fields with the volume form of $\R^{2k+1}$. In the 3-dimensional case, we take a single vector field $V$ and construct a dual 2-form $\alpha$ by contracting the vector field with the volume form of $\R^3$ according to the rule
\begin{equation*}
\alpha(W_1,W_2) = \dVol(V,W_1,W_2).
\end{equation*}
Under this correspondence between vector fields and forms, a divergence-free vector field tangent to the boundary of $\Omega$ becomes a closed Dirichlet $2$-form.

In general, we begin with a closed Dirichlet $(k+1)$-form $\alpha$, defined on a domain $\Omega$ in $\R^{2k+1}$.  Our first goal is to define the helicity of $\alpha$ in terms of the cohomology class represented by a form constructed from $\alpha$ in the configuration space $C_2[\Omega]$ of two disjoint points in $\Omega$. We will start by recalling the definition of $C_2[\Omega]$, which will be a smooth closed manifold with boundary and with corners, in Section~\ref{configspaces}.  We will then give our construction of helicity in Section~\ref{redefining}.   Extending helicity to cases where $k$ is odd produces a nontrivial function; when $k$ is even, helicity extends to a function that is always zero.  Finally, we will prove that our helicity is the standard helicity of a vector field in $\R^3$ in Section~\ref{helicitysame}.

\subsection{The Fulton--MacPherson compactification of configuration spaces}
\label{configspaces}

We start with a piece of technology: the Fulton-MacPherson compactification of a configuration space. There are many versions of this classical material (see for instance~\cite{MR1259368,MR1258919}). We follow Sinha~\cite{MR2099074} as this gives a geometric viewpoint appropriate to our setting. 

\begin{definition}
\label{cnmo}
Given an $m$-dimensional manifold $M$, define the \emph{configuration space} $\cnmo$ to be the subspace of $n$-tuples $(x_i) := (x_1, \dots, x_n) \in M^n$ such that $x_i \neq x_j$ if $i\neq j$. Let $\iota$ denote the inclusion of $\cnmo$ in~$M^{n}$. 
\end{definition}

\newcommand{\otuple}[2]{\genfrac{[}{]}{0pt}{}{#1}{#2}}
The configuration space $\cnmo$ may be thought of as the space of ordered $n$-tuples in $M^n$, without the diagonals.  Given $n$ distinct points in $M$, there are $n!$ points in $\cnmo$ corresponding to the permutations of the $n$ points. The Fulton--Macpherson compactification of $\cnmo$, defined below, keeps track of the directions and relative rates of approach when configuration points come together. To simplify the definition of $\cnm$, we introduce a bit of notation. Let $\otuple{n}{k}$ be the set of ordered $k$-tuples chosen from a set of $n$ elements and $\#\otuple{n}{k}$ to be the number of such tuples.

\begin{definition}[Sinha Definition 1.2]
For $(i,j)\in \otuple{n}{2}$, let $\pij\co C_n(\R^m)\rightarrow S^{m-1}$ be the map which sends~$(x_\ell)$ to the unit vector in the direction of $x_i-x_j$. Let $[0,\infty]$ be the one-point compactification of $[0,\infty)$.  For $(i,j,k)\in \otuple{n}{3}$, let $\sijk \co C_n(\R^m)\rightarrow [0,\infty]$ be the map which sends $(x_\ell)$ to $|x_i-x_j|/|x_i-x_k|$. 
\end{definition}

We work with an arbitrary smooth manifold $M$ by first embedding $M$ in~$\R^m$ (Whitney embedding), then defining the maps $\pij$ and $\sijk$ by restriction. Thus $C_n(M)$ is a submanifold of $C_n(\R^m)$. If $M=\R^m$, then it is a submanifold of itself through the identity map.  

\begin{definition} [Sinha Definition 1.3] Let $A_n[M]$ be the product 
\begin{equation*}
A_n[M] = (M)^{n}\times (S^{m-1})^{\#\otuple{n}{2}} \times [0,\infty]^{\#\otuple{n}{3}}.
\end{equation*}  
Define the \emph{Fulton--MacPherson compactification} $\cnm$ to be the closure of the image of $\cnmo$ under the map
$\iota_n= \iota \times (\pij\vert_{\cnmo}) \times ((\sijk)\vert_{\cnmo}) \co C_n(M)\rightarrow A_n[M]$. 
Let $\bdry \cnm:=\cnm - \cnmo$ denote the boundary of $C_n[M]$, the points that are added in the closure.
\end{definition}

We list a few properties of this compactification, from \cite{MR2099074} and Theorem~2.3 of~\cite{MR2102844}.
\begin{theorem}
\label{thm:config}
The spaces $\cnm$ and $\cnmo$ have the following properties:
\begin{itemize}
\item $\cnm$ is a ``manifold with corners'' with interior $\cnmo$. It has the same homotopy type as $\cnmo$. It is independent of the embedding of $M$ in $\R^m$, and it is compact if $M$ is.
\item \label{surjection} The inclusion of $\cnmo$ in $(M)^n$ extends to a surjective map $p$ from $\cnm$ to $(M)^n$ which is a homeomorphism over points in $\cnmo$. 
\item The boundary of $\cnm$ is stratified into a collection of faces of various dimensions.
\item An embedding $f \co M \rightarrow N$ induces an embedding of manifolds with corners called the evaluation map $ev_n[f] \co C_n[M]\rightarrow C_n[N]$ which respects the stratifications on the boundaries. 
\end{itemize}
\end{theorem}

The stratification of boundary faces of $\cnm$ has a beautiful combinatorial structure: in general, the set of faces of all codimensions is a Stasheff associahedron. While this structure is very interesting, we will use very little of it below, so we do not describe it in detail. We will need only the following:

\begin{lemma}
If $\Omega$ is a $k$-manifold with boundary embedded in $\R^n$ then the boundary of the manifold $C_2[\Omega]$ has three smooth faces, $\bdy\Omega \cross \Omega$, $\Omega \cross \bdy{\Omega}$ and an ``interior'' face diffeomorphic to the unit tangent bundle $UT(\Omega)$ of $\Omega$. Following the notation of Sinha, we will call this last face $(12)$, meaning that points 1 and 2 come together on that face. These codimension-1 faces of the boundary of $C_2[\Omega]$ meet at faces of higher codimension.
\end{lemma}

\begin{proof}
The space $C_2[\Omega]$ is a closed subspace of the larger space $A_2[\R^n]$ created by closing the image of $\Omega \cross \Omega$ under the map $\iota$. In this larger space, the boundary of the image consists of the image of $\bdy(\Omega \cross \Omega)$ together with a new boundary face created by removing the diagonal of $\Omega \cross \Omega$. 

We are only taking configurations of pairs of points in $\Omega$, so there are no $s_{ijk}$ maps, and only two $\pi_{ij}$ maps: $\pi_{12}$ and $\pi_{21}$. Along the new boundary face, then, the map $\iota$ records the location $z$ and limiting  direction $u$ of approach of pairs of points in $M$. This direction, recorded by $\pi_{12}$ and $\pi_{21}$, is a unit vector in the tangent space to $\Omega$ at~$z$. The set of all such pairs is a copy of $UT(\Omega)$.

These boundary faces meet at pairs of points where, for instance, an interior point approaches a boundary point, or where both points in the pair are on the boundary of $\Omega$. Sinha shows that these are faces of higher codimension.
\end{proof}

\begin{corollary} 
\label{c2obdy}
If $\Omega$ is a domain in $\R^{2k+1}$ with smooth boundary, then the boundary of $C_2[\Omega]$ consists of three faces diffeomorphic to $\Omega \cross \bdy\Omega$, $\bdy\Omega \cross \Omega$ and $\Omega \cross S^{2k}$.
\end{corollary}

Sinha additionally defines configuration spaces where one or more points in the configuration are fixed. In this case, $C_{n,k}[M]$ is the space of $n$ points where $k$ points are fixed and the remaining $n-k$ points vary.

\subsection{Redefining helicity} 
\label{redefining}

Motivated by the Bott-Taubes approach \cite{MR1295465} to defining finite-type knot invariants, we now seek to define helicity for a $(k+1)$-form $\alpha$ on a domain in $\R^{2k+1}$ by integration over an appropriate configuration space. We will construct a ``universal'' $2k$-form on $C_2[\R^{2k+1}]$ by a Gauss map and then define helicity to be the integral of the wedge product of the universal form and a form derived from~$\alpha$ over $C_2[\Omega]$. The corresponding approach for knots is explained beautifully by Volic~\cite{MR2300426}.

So let $\Omega \subset \R^{2k+1}$ be a compact subdomain with piecewise smooth boundary and let $\alpha$ be a closed Dirichlet $(k+1)$-form on $\Omega$.  In~$\R^3$, we may equivalently start with a smooth vector field $V$ on $\Omega$ that is divergence-free and tangent to the boundary, and take $\alpha$ to be the dual 2-form to $V$.  The divergence-free condition implies that $\alpha$ is closed, while the boundary condition on $V$ implies that $\alpha\vert_{\bdy\Omega}=0$.

\begin{lemma} \label{alpha} 
The closed $(k+1)$-form $\alpha$ on $\Omega$ pulls back to a pair of closed $(k+1)$-forms $\alpha_x$ and $\alpha_y$ on $C_2[\Omega]$.  Hence, their wedge product $\alpha_x \wedge \alpha_y$ is a closed $(2k+2)$-form on $C_2[\Omega]$. 
\end{lemma}

\begin{proof} 
We take the surjective map $p: C_2[\Omega] \onto \Omega \cross \Omega$, guaranteed by Theorem~\ref{thm:config}, and  compose it with either of the two projections from $\Omega \cross \Omega \onto \Omega$ to obtain a map $C_2[\Omega] \onto \Omega$.  If we take $(x,y) \in C_2[\Omega]$, then these maps send $(x,y) \mapsto x$ or $(x,y) \mapsto y$. The pullback of $\alpha$ under the first map will be denoted $\alpha_x$ and the pullback under the second will be denoted $\alpha_y$. Since $\alpha$ is closed, $\alpha_x$, $\alpha_y$ and $\alpha_x \wedge \alpha_y$ are all closed forms.
\end{proof}

We now want to study the pullback of $\alpha_x \wedge \alpha_y$ to the boundary of $C_2[\Omega]$. To do so, we must first introduce coordinates on that boundary. As $\ctwoo$ has codimension $4k$ in the ambient space $\ao$, the $8k+2$ natural coordinates on $\ao$ ($2k+1$ on $\Omega_x$, $2k+1$ on $\Omega_y$, $2k$ on each $S^{2k}$) overdetermine coordinates on $\ctwoo$.  In a neighborhood near the ``interior'' boundary face $(12)$, which we recall is diffeomorphic to $\Omega \cross S^{2k} \subset \ctwoo$ by Corollary~\ref{c2obdy}, it will be convenient to work with three different coordinate systems:
\begin{itemize}
\item {\it configuration coordinates}:  $\{x_i, y_j\}$.  These induce well-defined values on $S^{2k}$ except on the face $(12)$ where $x=y$.
\item {\it midpoint-offset coordinates}:  $\{m_i, o_j\}$.  Define $m:=(x+y)/2$ to be the midpoint of $xy$ and $o:=(x-y)/2$ to be the offset between $x$ and $y$. These variables are defined so that $x = m + o$ and $y = m - o$. These also induce well-defined values on $S^{2k}$, except on $\{o=0\}$, which describes the boundary face $(12)$.
\item {\it boundary spherical coordinates}:  $\{z_i, r, u_j\}$.   Define $\{r, u_j\}$ as spherical coordinates on the $o_j$ variables above so that $u_j$ is always a unit vector and 
\begin{equation*}
r = \norm{o}.
\end{equation*}
These have the advantage of naturally extending to the boundary face $(12)$, described by $\{r=0\}$.   
\end{itemize}

On $(12)$, the boundary spherical coordinates provide natural coordinates $\{z_i,  u_j\}$.  The $\{z_i\}$ describe the point $x=y$ while the $u_j$ measure the limiting direction by which $x$ and $y$ approached each other.

\begin{lemma}
\label{alphavanishesonbdy}
If $\alpha$ is a Dirichlet form on $\Omega$, then $\alpha_x \wedge \alpha_y$ is a Dirichlet form on $C_2[\Omega]$.
\end{lemma}

\begin{proof}
As we saw in Corollary~\ref{c2obdy}, the boundary of $C_2[\Omega]$ consists of three codimension one faces: $\bdy\Omega \cross \Omega$, $\Omega \cross \bdy\Omega$ and $(12)$. On the first two boundary faces, either $x$ or $y$ is on $\bdy\Omega$.  But $\alpha$ vanishes when pulled back to $\bdy\Omega$, so $\alpha_x$ vanishes on $\bdy\Omega_x$ and $\alpha_y$ on $\bdy\Omega_y$.  Thus $\alpha_x \wedge \alpha_y$ vanishes on these faces.

The third codimension one face, which we call face $(12)$, is all that remains. For convenience, let $I=(i_1, \ldots, i_k)$ denote a multi-index, so that $dx_I = dx_{i_1} \wedge \dots \wedge dx_{i_k}$. Using this notation, we observe that $\alpha_x$ can only consist of terms such as $h_{I}(x) dx_I$, with no $y$ dependence.  Similarly, $\alpha_y$ consists of terms with the same coefficient functions $h_{I}(y) dy_I$.  The functions $h_{I}(x)$ are smooth functions of $x$ since our original 2-form $\alpha$ on $\Omega$ was smooth.

Consider these terms on the boundary face $(12)$, which is a copy of $\Omega \cross S^{2k}$.  We will write $\alpha_x \wedge \alpha_y$ in the boundary spherical coordinates $\{z_i, u_j \}$. In ``midpoint-offset'' coordinates, $x=m+o$ and $y=m-o$.  Thus each $dx_i = dm_i + do_i$. If we now convert to boundary spherical coordinates using $o = r u$, then we see that $do_i = u_i dr + r du_i$. Now on the boundary face $(12)$, we have $r = 0$. 
The coefficient functions $h_{I}$ are smooth at $r=0$, so the term $h_{I} r du_I = 0$ on $(12)$ and the pullback of $\alpha_x$ to the boundary can have no $du_i$ terms. Further, $dr$ vanishes when pulled back to the boundary $S^{2k}$, so no $dr$ terms can be involved either. This means that the $(k+1)$-form $\alpha_x$ is expressed entirely in terms of the $2k+1$ midpoint coordinates $dm_i$.

But the same is true for $\alpha_y$, so the $(2k+2)$-form $\alpha_x \wedge \alpha_y$ involves only the $2k+1$ elementary 1-forms $dm_i$. Thus some $dm_i$ is repeated, forcing this form to be zero.
\end{proof}

In the original definition of helicity in~\eqref{mhelicity}, we integrated over $\Omega \cross \Omega$ even though the integrand was not defined on the diagonal. To justify the integration, it would be enough to show that the integrand converged on the diagonal. In fact, we can show that the integrand vanishes as we approach the diagonal. Lemma~\ref{alphavanishesonbdy} is the appropriate version of that familiar statement in our new setting.

We now give a definition:
\begin{definition}
\label{gaussmap}
The \emph{Gauss map} $g \co C_2(\Omega) \rightarrow S^{2k}$ is given by $(x,y) \mapsto (y-x)/\norm{y-x}$.
\end{definition}

\begin{lemma}
\label{gausslem}
The Gauss map is a smooth map defined on all of $C_2[\Omega]$, including the boundary. The pullback of the unit volume form $\Vol$ on $S^{2k}$ by $g$ defines a closed $2k$-form $g^*\dVol$ on $C_2[\Omega]$.
\end{lemma}
 
\begin{proof}
The Gauss map extends naturally to $\Omega \times \bdy\Omega$ and $\bdy\Omega \times \Omega$, so we only have to worry about the boundary face $(12)$ of $C_2[\Omega]$.  But by construction, $(12)$ is a blow-up of the diagonal of $\Omega \times \Omega$ so that the maps $\pi_{ij}$ extend smoothly to the boundary.  In this case, $\pi_{21} = g$, so the lemma is proven.
\end{proof}

This lemma demonstrates why $C_2[\Omega]$ was better for our construction than $\Omega \cross \Omega$. While the latter is simpler to work with, we could not have extended the Gauss map smoothly to the diagonal of $\Omega \cross \Omega$. In fact, the form $g^*\dVol$ is the same as the ``linking form'' of~\cite{MR1028280} (which is defined on $\Omega \cross \Omega$) on the interior of $C_2[\Omega]$. The essential difference is that $g^*\dVol$ extends smoothly to the boundary of $C_2[\Omega]$ while the linking form has a singularity on the diagonal of $\Omega \cross \Omega$.

We can now combine the observations of Lemmas~\ref{alpha}, \ref{alphavanishesonbdy}, and~\ref{gausslem} to redefine helicity.  We have shown that if $\alpha$ is a closed Dirichlet form on $\Omega$,  then $\alpha_x \wedge \alpha_y$ is a closed Dirichlet form on $C_2[\Omega]$.  Hence $\alpha_x \wedge \alpha_y$ represents a relative de Rham cohomology class $[\alpha_x \wedge \alpha_y]$ in $H^{2k+2}(C_2[\Omega],\bdy C_2[\Omega]; \R)$. Similarly, $g^*\dVol$ is closed 
so it represents an absolute de Rham cohomology class $[g^*\dVol]$ in $H^{2k}(C_2[\Omega]; \R)$. We will use de Rham cohomology (and thus coefficients in $\R$) for the rest of the paper. We now make an observation about the volume form on $C_2[\Omega]$:

\begin{lemma}
\label{lem:volumeform}
If $M$ has a volume form $\dVol_M$, then there is a natural volume form $\dVol_{C_2[\Omega]}$ with total volume $\Vol(C_2[\Omega]) = \Vol(\Omega)^2$.
\end{lemma}

\begin{proof}
Just as we pulled back the $(k+1)$-form $\alpha$ to forms $\alpha_x$ and $\alpha_y$ on $C_2[\Omega]$, we can pull back $\dVol_\Omega$ to $(\dVol_\Omega)_x$ and $(\dVol_\Omega)_y$. Then $\dVol(C_2[\Omega]) = (\dVol_{\Omega})_x \wedge (\dVol_{\Omega})_y$.
\end{proof}
This lemma enables us to define helicity.
\begin{definition}
\label{fthelicity}
If $\alpha$ is a closed Dirichlet $(k+1)$-form on $\Omega \subset \R^{2k+1}$, then we have seen that $\alpha_x \wedge \alpha_y$ defines a cohomology class in $H^{2k+2}(C_2[\Omega],\bdy C_2[\Omega])$. We also know that $g^*\dVol_{S^{2k}}$ defines a cohomology class in $H^{2k}(C_2[\Omega])$. Let $[\dVol_{C_2[\Omega]}] \in H^{4k+2}(C_2[\Omega],\bdy C_2[\Omega]) \simeq \R$ be the top class of $C_2[\Omega]$ defined by the standard volume form. The cup product $[\alpha_x \wedge \alpha_y] \cup [g^*\dVol_{S^{2k}}]$ is in $H^{4k+2}(C_2[\Omega])$ and is hence a multiple of $[\dVol_{C_2[\Omega]}]$.

We define the \emph{helicity} $\Hel(\alpha)$ of $\alpha$ by 
\begin{equation*}
[\alpha_x \wedge \alpha_y] \cup [g^*\dVol] = \frac{\Hel(\alpha)}{\Vol(\Omega)^2} [\dVol_{C_2[\Omega]}].
\end{equation*}
We can calculate $\Hel(\alpha)$ explicitly as the integral
\begin{equation} \label{hdefn}
\Hel(\alpha) = \int_{C_2[\Omega]} \alpha_x \wedge \alpha_y \wedge g^*\dVol_{S^{2k}}.
\end{equation}
\end{definition}
Let $\Phi=\alpha_x \wedge \alpha_y \wedge g^*\dVol_{S^{2k}}$ denote the integrand above.

In Theorem~\ref{ftclassical}, we will show that our definition agrees with the classical integral on three-dimensional domains and the usual extension to the helicity of $(k+1)$-forms on $(2k+1)$-dimensional domains.  As we expect from the theory of the Hopf invariant~\cite[Proposition 17.22]{MR658304},
\begin{proposition}
\label{even-k}
For even $k$ values, the helicity of every $(k+1)$-form is zero.
\end{proposition}

\begin{proof}
Let us consider the automorphism $a$ of $C_2(\Omega)$ that interchanges $x$ and $y$; it extends naturally to $C_2[\Omega]$.  It reverses the orientation of $C_2[\Omega]$, since it exchanges the order of a product of odd-dimensional spaces.

We take the pullback $a^*\Phi = \alpha_y \wedge \alpha_x \wedge a^*g^*\dVol_{S^{2k}}$.  The map $a$ induces an antipodal map on $S^{2k}$; such a map has degree $-1$.  Hence, $a^*g^*\dVol_{S^{2k}}= -g^*\dVol_{S^{2k}}$.  Also, $\alpha_y \wedge \alpha_x = (-1)^{(k+1)^2} \alpha_x \wedge  \alpha_y$.  Combining these results, $a^*\Phi = (-1)^k \Phi$.  We then compute
\begin{equation*}
-\Hel(\alpha) = \int_{-C_2[\Omega]} \Phi = \int_{a\left(C_2[\Omega]\right)} \Phi = \int_{C_2[\Omega]} a^*\Phi = \int_{C_2} (-1)^k \Phi = (-1)^k \Hel(\alpha).
\end{equation*}

If $k$ is even, this implies that $\Hel(\alpha) = -\Hel(\alpha)$, i.e., that helicity is zero, and proves our proposition.  If $k$ is odd, the conclusion is a tautology: $\Hel(\alpha) = \Hel(\alpha)$.
\end{proof}

\subsection{Comparison with the standard definition of helicity}
\label{helicitysame}

This description of helicity as a cohomology class may seem quite different from the definition of helicity that we gave earlier. So before we explore the consequences of our new definition, we will reassure ourselves that this approach is correct by showing explicitly that for 2-forms defined on domains in $\R^3$, our 6-form $\Phi$ on $C_2(\Omega)$ is exactly the classical helicity integrand. 

\begin{lemma}
\label{lemma-int}
Let $\Omega$ be a compact subdomain of $\R^3$ with smooth boundary, and let $\alpha$ be a closed Dirichlet 2-form on $\Omega$.  Let $V$ be the vector field dual to $\alpha$.  Recall from Definition~\ref{cnmo} that the map $\iota$ naturally embeds $C_2(\Omega)$ into $\Omega \times \Omega$.  Then, the integrand $\Phi$ from (\ref{hdefn}), namely the 6-form $\alpha_x \wedge \alpha_y \wedge g^*\dVol $, is equal to the pullback via $\iota$ of the classical helicity integrand
\begin{equation}
\label{hintegrand}
\frac{1}{4\pi} V(x) \times V(y) \cdot \frac{x-y}{|x-y|^3} \; \dVol_x \, \dVol_y.
\end{equation}
\end{lemma}

With the lemma in place, we now conclude that our definition of helicity really is the same as the standard one.
\begin{theorem}
\label{ftclassical}
For three-dimensional domains, the helicity of Definition~\ref{fthelicity}, equals the classical helicity (of equation~\ref{mhelicity}).  More explicitly, for a vector field $V$ dual to a 2-form~$\alpha$,
\begin{equation} \label{eqhelicity}
\int_{C_2[\Omega]} \alpha_x \wedge \alpha_y \wedge g^*\dVol_{S^2} = \frac{1}{4\pi} \int_{\Omega \times \Omega} {V(x) \times V(y) \cdot \frac{x-y}{|x-y|^3} \; \dVol_x \, \dVol_y}
\end{equation}
\end{theorem}

\begin{proof}
The idea of the proof is to remove small neighborhoods of the boundary of $\ctwoo$ and of the diagonal 
$\Delta$ in $\Omega \times \Omega$.  On the removed neighborhoods, the integrals each tend to zero.  On what remains, one integral is simply the pullback of the other.

Denote the integrand (\ref{hintegrand}) as $\mu$.  Let $U_\epsilon$ be an $\epsilon$-neighborhood of    $\partial \ctwoo$.  Then,
\[  \int_{C_2[\Omega]} \Phi  =  \int_{C_2[\Omega] - U_\epsilon } \Phi  \; + \; \int_{U_\epsilon} \Phi \]
As $\epsilon \rightarrow 0$, so does $\int_{U_\epsilon} \Phi$.  Then, the above lemma guarantees that $\Phi=\iota^*\mu$ on $C_2[\Omega] - U_\epsilon$.  Hence,
\[\int_{C_2[\Omega] - U_\epsilon} \Phi \; = \; \int_{C_2[\Omega] - U_\epsilon} \iota^*\mu \; = \; \int_{\iota(C_2[\Omega] - U_\epsilon )} \mu \]
But, the image $\iota(C_2[\Omega] - U_\epsilon )$ is $\Omega \times \Omega$ with some neighborhood $V_\epsilon$, dependent upon~$\epsilon$, removed.  As $\epsilon \rightarrow 0$, the set $V_{\epsilon}$ approximates $\Delta$.  

While the integral $\int_{\Omega \times \Omega} \mu$ is improper along the diagonal, it does in fact converge.  The contribution of $\mu$ integrated over neighborhoods of the diagonal converges to 0.  See \cite{MR1770976} for details.

Hence,  $\int_{\iota(\cnmo - U_\epsilon )} \mu$ limits to the classical helicity integral $\int_{\Omega \times \Omega} \mu$.  But it also limits to $\int_{C_2[\Omega]} \Phi$, so the two are equal.
\end{proof}

We now prove the above lemma in local coordinates at an arbitrary point in $C_2(\Omega)$.  
\begin{proof}[Proof of Lemma~\ref{lemma-int}]
A choice of coordinates on $\Omega$ induces a set of configuration coordinates on $C_2(\Omega)$.  At the point $p=(x,y) \in C_2(\Omega)$, we choose right-handed orthonormal coordinates $\{u_i\}$ on $\Omega$ so that $u_3$ points along the vector $y-x$ at $p$.  Via the map $\iota$ from Definition~\ref{cnmo}, these induce coordinates $\{x_i, y_i\}$ on $C_2(\Omega)$.  We now calculate $\Phi$ and the classical helicity integrand \eqref{hintegrand} in these coordinates at $p$.

Begin by writing 
\begin{equation*} 
V(x) = v_1 \ddu{1} + v_2 \ddu{2} + v_3 \ddu{3} \qquad \mathrm{and}  \qquad V(y) = w_1 \ddu{1} + w_2 \ddu{2} + w_3 \ddu{3}
\end{equation*}
so that
\begin{align*}
\alpha_x & = v_1 \; dx_2 \wedge dx_3 \; + \; v_2 \;dx_3 \wedge dx_1  \; + \;  v_3 \; dx_1 \wedge dx_2,  \\    \alpha_y &= w_1 \; dy_2 \wedge dy_3  \; + \;  w_2\;  dy_3 \wedge dy_1  \; + \;  w_3 \; dy_1 \wedge dy_2. 
\end{align*}
By the choice of coordinates, $\ds{\frac{x-y}{|x-y|^3} = - \frac{1}{|x-y|^2} \; \ddu{3}}$.  Then, the classical helicity integrand is
\begin{equation}
\label{int-coords}
\frac{1}{4\pi} V(x) \times V(y) \cdot \frac{x-y}{|x-y|^3} \dVol_x \, \dVol_y = \frac{1}{4\pi} \frac{1}{|x-y|^2} (v_2 w_1 - v_1 w_2) \dVol_x \, \dVol_y.
\end{equation}

Now we calculate $\Phi$ in these coordinates; we start with $g^*\dVol$, the pullback of the unit area form on $S^2$ via the Gauss map.  Moving the configuration points in the $x_3$ (or $y_3$) direction, that is moving them closer or further apart, has no impact upon the Gauss map $g$, so $g^*\dVol$ contains no $dx_3$ or $dy_3$ terms.  Writing it in terms of the other 2-forms, we get
\begin{align*}
g^*\dVol  = & c_1 dx_1 \wedge dx_2  \; + \; c_2 dy_1 \wedge dy_2  \; + \;  c_3 dx_1 \wedge dy_1  \; + \; c_4 dx_2 \wedge dy_2 \\
& + c_5 dx_1 \wedge dy_2  \; + \;  c_6 dy_1 \wedge dx_2.
\end{align*}

So which bi-vectors on $C_2(\Omega)$ span area on $S^2$ under $g$?  Neither $\ddx{1} \wedge \ddy{1}$ nor $\ddx{2} \wedge \ddy{2}$ does, so $c_3 = c_4 = 0$.  The other bi-vectors do; their effect must be normalized by the distance squared between $x$ and $y$ and also by the fact that the area of the sphere integrates to $1$  (we are using the unit area form $\dVol$).  Considering orientations, $c_1= c_2 = 1/4\pi|x-y|^2 = -c_5 = -c_6$.  So
\begin{equation*}
g^*\dVol  = \frac{1}{ 4\pi|x-y|^2} \left(dx_1 \wedge dx_2  \; + \; dy_1 \wedge dy_2  \; - \; dx_1 \wedge dy_2  \; - \;   dy_1 \wedge dx_2 \right).
\end{equation*}
We compute the 6-form $\Phi$ to be  $c_5 v_1 w_2 - c_6 v_2 w_1$, which by substituting becomes
\begin{align}
\Phi & =  \frac{1}{4\pi}\frac{1}{\norm{x-y}^2} (v_2 w_1 - v_1 w_2) \; dx_1 \wedge dx_2 \wedge dx_3 \wedge dy_1 \wedge dy_2\wedge dy_3 \label{phi-coords}
\end{align}
Pulling the classical helicity integrand (\ref{int-coords}) back via $\iota$, we obtain $\Phi$ since  $\iota^*(\dVol_x) = dx_1 \wedge dx_2 \wedge dx_3$ (and similarly for $\dVol_y$).
\end{proof}

\section{Understanding the properties of helicity via cohomology}

We have now defined helicity as a cup product of cohomology classes and have shown in the case of vector fields in $\R^3$ that our definition is the standard helicity integral. We now consider the consequences of our new definition and try to provide some motivation for the definition now that we have made it.

\subsection{Invariance of helicity under diffeomorphisms homotopic to the identity}

In the Introduction, we discussed the development of the Helicity Invariance theorem, from the earliest versions of helicity as an invariant of ideal MHD through Arnold's picture of helicity as invariant under all diffeomorphisms on simply-connected domains to the modern picture of helicity as invariant under diffeomorphisms which are homotopic to the identity.  Our redefinition of helicity allows us to give a quick proof of this invariance result.

\begin{proposition}
\label{homotopic-invariance}
Let $\Omega$ be any domain in $\R^{2k+1}$ and let $\alpha$ be a closed Dirichlet $(k+1)$-form on $\Omega$.

Let $f\co\Omega \cross I \rightarrow \R^{2k+1}$ be a smooth map. For each fixed $t$, define $f_t\co \Omega \rightarrow \Omega_t \subset \R^{2k+1}$ by $f_t(p) = f(p,t)$ and assume that each $f_t$ is a diffeomorphism, with $f_0$ the identity map. Let $\alpha_t = (f_t^{-1})^* \alpha$ on each $\Omega_t$.

Then $\Hel(\alpha)$ on $\Omega_0 = \Omega$ is equal to $\Hel(\alpha_1)$ on $\Omega_1$.
\end{proposition}

\begin{proof}
There is a natural projection $\Omega \cross I \rightarrow \Omega$ given by $(p,t) \mapsto p$. Pulling back under this map, we can extend $\alpha$ to a form on $\Omega \cross I$. Similarly, there are two obvious projections $\pi_x, \pi_y \co C_2[\Omega] \cross I \rightarrow \Omega \cross I$ given by $(x,y,t) \mapsto (x,t)$ and $(x,y,t) \mapsto (y,t)$. Pulling back under these maps, we can define a closed form $\alpha_x \wedge \alpha_y$ on $C_2[\Omega] \cross I$.

We next define an extended Gauss map on $C_2[\Omega] \cross I$ by 
\begin{equation*}
G(x,y,t) = \frac{f_t(x) - f_t(y)}{\norm{f_t(x) - f_t(y)}}.
\end{equation*}
This map allows us to construct a closed $2k$-form $G^*\dVol_{S^{2k}}$ on $C_2[\Omega] \cross I$.  We note that the $(4k+2)$ helicity form $\alpha_x \wedge \alpha_y \wedge G^*\dVol_{S^{2k}}$ is a closed Dirichlet form (by Lemma~\ref{alphavanishesonbdy}) on $C_2[\Omega]$.  By Stokes' theorem, the integral of this form over $\bdy (C_2[\Omega] \cross I)$ is zero.  But this means that 
\begin{equation}
\label{topbottom}
\int_{C_2[\Omega] \cross \{0\}} \alpha_x \wedge \alpha_y \wedge G^*\dVol_{S^{2k}} = \int_{C_2[\Omega] \cross \{ 1 \}} \alpha_x \wedge \alpha_y \wedge G^*\dVol_{S^{2k}}.
\end{equation}
We now prove that the left hand side is $\Hel(\alpha_0)$ and the right hand side is $\Hel(\alpha_1)$.

Since $f_0$ is the identity map, $G(x,y,0) = g(x,y)$ and the left hand side is clearly $\Hel(\alpha) = \Hel(\alpha_0)$.  On the right-hand side, we observe that by definition
\begin{equation*}
\Hel(\alpha_1) = \int_{C_2[\Omega_1]} (\alpha_1)_x \wedge (\alpha_1)_y \wedge g^*\dVol_{S^{2k}} 
               = \int_{C_2[\Omega_1]} (F^{-1})^*\alpha_x \wedge (F^{-1})^*\alpha_y \wedge g^*\dVol_{S^{2k}}.
\end{equation*}
where $F \co C_2[\Omega] \rightarrow C_2[\Omega_1]$ is the map of configuration spaces induced by $f_1$ (c.f., Theorem~\ref{thm:config}).  We note that $G(x,y,1)=g \circ F(x,y)$.    If we pull back the integral above to $C_2[\Omega] \cross \{1\}$ using $F^{-1}$, we get the right hand side of \eqref{topbottom}.
\begin{align*}
\Hel(\alpha_1) & = \int_{F^{-1}(C_2[\Omega_1])=C_2[\Omega]} \alpha_x \wedge \alpha_y \wedge F^*g^*\dVol_{S^{2k}}  \\
& = \int_{C_2[\Omega] \cross \{ 1 \}} \alpha_x \wedge \alpha_y \wedge G^*\dVol_{S^{2k}}.  \qedhere
\end{align*}
\end{proof}

\subsection{The invariance theorems for helicity and finite-type invariants}

We could have proved this theorem in a new way, parallel to the proof of invariance for the finite-type invariants for knots. Let $\Embed$, henceforth denoted~$\Emb$, consist of all diffeomorphic embeddings of $\Omega$ into $\R^{2k+1}$.  Maps in each connected component of $\Emb$ are diffeotopic to one another. Define a Gauss map $g_f$ by 
\begin{equation*}
((x,y),f) \in C_2[\Omega] \times \mathfrak{E} \mapsto \displaystyle{\frac{f(x)-f(y)}{|f(x)-f(y)|}}.
\end{equation*}

Consider the following diagram:
\begin{equation}
\label{embedbundle}
\begin{CD}
C_2[ \Omega] \times \Embed  @>g_f>>  S^{2k} \\
@VV{\pi}V						  \\
\Embed				
\end{CD}
\end{equation}
where $\pi$ is the natural projection in the trivial bundle $C_2[\Omega] \cross \Embed \rightarrow \Embed$. This is analogous to the corresponding diagram for knots introduced by Bott and Taubes~\cite{MR1295465}.

Define the $(4k+2)$-form 
\begin{equation}
\label{defgamma}
\Phi = \alpha_x \wedge \alpha_y \wedge g_f^{*}(\dVol)
\end{equation}
on $C_2[\Omega] \cross \Emb$ by pulling back $\alpha_x \wedge \alpha_y$ from $C_2[\Omega]$ and the volume form $\dVol$ from $S^{2k}$; we note that $\Phi$ is Dirichlet, by Lemma~\ref{alphavanishesonbdy}, and is closed.  We now observe that integration of $\Phi$ over the fiber in the bundle $C_2[\Omega] \cross \Emb \rightarrow \Emb$ produces a $0$-form $\Hel(f)$ on $\Emb$. The value of this $0$-form on any embedding is the helicity $\Hel((f^{-1})^*\alpha)$.

Using Stokes' Theorem, we compute
\begin{equation*}
d \Hel(f) = d\int_{C_2[\Omega]}{\Phi} = \int_{C_2[\Omega]}{d\Phi} - \int_{\bdy C_2[\Omega]}{\Phi} = 0 - 0,
\end{equation*}
since $\Phi$ is a closed Dirichlet form.  Since $d\Hel(f) = 0$, we conclude that $\Hel(f)$ is constant on each connected component of $\Emb$.  This reproves that helicity is invariant under diffeomorphisms homotopic to the identity.

\subsection{Invariance of helicity for forms and vector fields}

The original invariance theorem for helicity of vector fields (Theorem~\ref{classicalinvariance}) required that the diffeomorphisms be volume-preserving. Our theorems about the invariance of the helicity of forms, by contrast, have no such requirement. 

If we fix our attention on the case $2k+1 = 3$, and consider the duality between 2-forms and vector fields, we immediately observe where the volume-preserving condition arises.  Start with $V$ dual to $\alpha$ on $\Omega$ and a diffeomorphism $f$ that lies in the same component of $\E$ as the identity. The helicity of $\alpha$ on $\Omega$ is the same as the helicity of the 2-form $\tilde{\alpha} = (f^{-1})^* (\alpha)$ on $f(\Omega)$.  However if $f$ is not volume-preserving, $\tilde\alpha(\cdot, \cdot)$ may not be dual to the pushforward vector field $f_*V$ because the duality operation explicitly involves the volume form on $f(\Omega)$.  Hence, differential forms produce a stronger invariance than vector fields do.

\subsection{Invariance of helicity defined with cohomologous forms}

Another interesting feature of Definition~\ref{fthelicity} is that the helicity of $\alpha$ depends only on the cohomology classes of $[\alpha_x \wedge \alpha_y]$ and $[g^*\dVol_{S^{2k}}]$. In particular, this means that we may  define the helicity integrand using any volume form on $S^{2k}$ which integrates to $1$ over the sphere and get an alternate integral formula for helicity. We are motivated here by the combinatorial formula for linking number, which is derived from the Gauss integral formula for linking number by concentrating the mass of the sphere at the north pole. This gives us a recipe for constructing new helicity integrals.

\begin{definition}
Given a point $x = (x_1,x_2,x_3)$ in a domain $\Omega$ in $\R^3$, let $x^+(\Omega)$ be the set of points $y = (x_1,x_2,y_3) \in \Omega$ with $y_3 > x_3$.
\end{definition}
We then have
\begin{proposition}
\label{prop:combinatorial}
The helicity of a divergence-free vector field in $\R^3$ which is tangent to the boundary of a domain $\Omega$ is given by the $4$-dimensional integral
\begin{equation*}
\Hel(V) = \frac{1}{4\pi} \int_{x \in \Omega} \int_{y \in x^+(\Omega)} V(x) \cdot V(y) \cross (0,0,1) \; \dVol_x dy_3.
\end{equation*}   
\end{proposition}

\begin{proof}
Consider a sequence of $2$-forms on $S^2$ converging to the $\delta$-form which concentrates the area of the sphere at the north pole where each has integral $4\pi$ over the entire sphere. These forms are cohomologous as $2$-forms on $S^2$ to the standard area form, so their pullbacks generate cohomologous $2$-forms on $C_2[\Omega]$. This means that the helicities derived from the forms in the sequence are all equal to the standard helicity. But the limit of these integrals is the formula above.
\end{proof}

We now do an explicit helicity computation using the formula to check that it works. It is an old theorem of Moffatt~\cite{mof1} and Berger and Field~\cite{MR770136} that the helicity of a divergence-free field tangent to the boundary of a pair of linked tubes is equal to the helicity of the fields in each tube plus twice the linking number of the tubes multiplied by the square of the flux of the field in the tubes (see~\cite{MR2002f:53002} for a more general version of this theorem).  Imagine then, a pair of singly-linked tubes that have rectangular cross-section with width $w$ and height $h$ and one overcrossing and that contain unit length fields parallel to the walls.  We will assume that at the overcrossing the tubes are rectangular boxes in parallel planes, as below in Figure~\ref{helicityfig}.
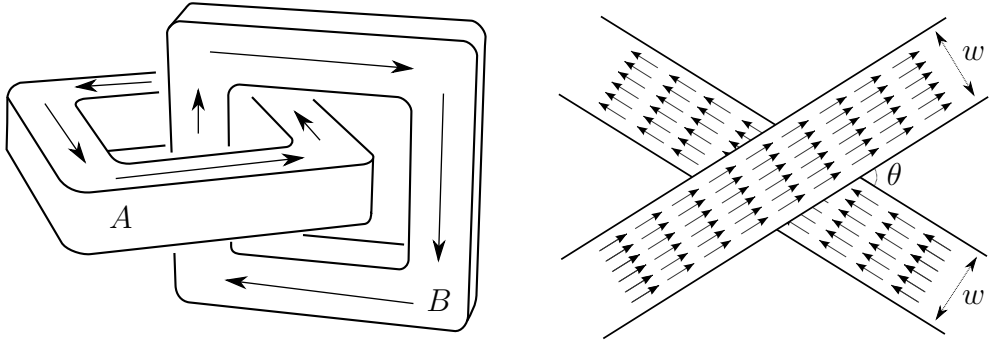
\begin{figure}[ht]
\begin{overpic}{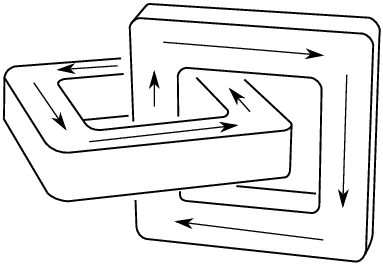}
\put(22,23){$A$}
\put(87,6){$B$}
\end{overpic}
\hspace{0.25in}
\begin{overpic}[height=1.7in]{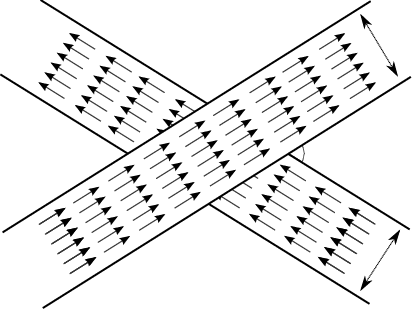}
\put(93.5,65){$w$}
\put(76,36){$\theta$}
\put(93.5,8.5){$w$}
\end{overpic}
\caption{A pair of singly linked tubes $A$ and $B$ with rectangular cross section and a unit-length vector field tangent to the boundary of the tubes. On the right, we see the crossing from above. The crossing appears twice in the four-dimensional integral of Proposition~\ref{prop:combinatorial}. We are able to explicitly compute the helicity of this configuration using the proposition and show that it is equal to twice the product of the fluxes of the field in the tubes.}
\label{helicityfig}
\end{figure}

We can arrange the tubes so that for any pair $\{x,y\}$ with $x$ and $y$ in the same tube and $y \in x^+(\Omega)$, the vectors $V(x)$ and $V(y)$ are collinear.  Since the integrand above vanishes for collinear vectors, these pairs will not contribute to the integral. We may further arrange the tubes so that there are only two regions where $x$ and $y$ are in different tubes and $y \in x^+(\Omega)$.  The overcrossing pictured is one region, with $x$ lying in the right side of ring $A$ and $y$ above it in the upper segment of ring $B$.  The other region has $x$ in the lower section of $B$ and $y$ in $A$. 

We now need to integrate over these pairs. The triple product in the integrand can be rewritten $(0,0,1) \cdot V(x) \cross V(y)$. Since $V(x)$ and $V(y)$ lie in horizontal planes in this region, the integrand always takes the constant value $\sin\theta$.  On the other hand, the domain of integration (for $x$) is a prism of height $h$ whose base is a parallelogram of length $w/\sin\theta$ and width $w$; the domain of integration for $y$ is a line segment above each point in the $x$ prism of length $h$.  Thus the total (4-dimensional) volume of integration is $w^2 h^2/\sin\theta$, and the value of the integral is $w^2h^2 = (wh)^2$. This is exactly the square of the flux of the vector field, and we note that the crossing is positively oriented.  The other region with $y \in x^+(\Omega)$ has $x$ in the lower section of tube $B$ and $y$ in tube $A$. This configuration is similar to the first, and makes the same contribution to the integral.

\section{Helicity as a wedge product with a primitive}

Arnol'd~\cite{MR891881} defines helicity for 2-forms on simply connected 3-manifolds as the integral of the wedge product of a form $\alpha$ and a primitive form $\beta$ with $d\beta = \alpha$.  In section~\ref{redefining}, we provided an alternate definition in terms of cohomology classes on configuration spaces.  In this section, we reconcile these two approaches.  Our efforts culminate in the next section with a formula for the change in helicity under an arbitrary diffeomorphism of $\Omega$.   

\subsection{Constructing a primitive form.}
We start by observing that there is a natural fiber bundle
\begin{equation}
\label{bsbundle}
\begin{CD}
C_{2,1}[ \Omega] @>i>>  C_2[\Omega] \\
@. @VV{\pi_x}V						  \\
@. \Omega		
\end{CD}
\end{equation}
where $\pi_x$ is the projection where $(x,y) \mapsto x$. Consider the $(k+1)$-form $\alpha_y = \pi_y^*\alpha$ from Lemma~\ref{alpha} generated by pulling back $\alpha$ from $\Omega$ in the corresponding projection $\pi_y$ where $(x,y) \mapsto y$, and the $2k$-form $g^*\dVol$ generated by pulling back the volume form on $S^{2k}$ under the Gauss map. We now develop some standard properties of this bundle.

\begin{definition}[cf.~\cite{MR1612569}, Definition 4.18]
We define the \emph{Biot-Savart operator for forms} to be the operation on $(k+1)$-forms $\alpha$ on $\Omega$ defined by the integration over the fiber in the bundle \eqref{bsbundle}, 
\begin{equation}
\BS(\alpha) = \frac{1}{\Vol(S^{2k})} \int_{C_{2,1}[\Omega]} \alpha_y \wedge g^*\dVol .
\end{equation}
\end{definition}

Proposition~\ref{alphaisexact} guarantees that any closed Dirichlet $(k+1)$-form $\alpha$ is exact.  We now show that the Biot-Savart operator constructs a primitive for $\alpha$.

\begin{proposition}[cf.~\cite{MR1612569}, Proposition 4.19]
\label{bsoperator}
If $\alpha$ is a closed Dirichlet $(k+1)$-form on $\Omega$, then $BS(\alpha)$ is a primitive for $\alpha$:
\begin{equation}
d(\BS(\alpha)) = \alpha.
\end{equation}
\end{proposition}

\begin{proof}
We will use Stokes' Theorem for fiber bundles $F^n \rightarrow E \rightarrow B$ (see Appendix~\ref{gv}). If $\beta$ is a $k$-form on $E$, then $\int_F \beta$ is a $(k-n)$-form on $B$, and 
\begin{equation} \label{stokes}
d \int_F \beta = \int_F d\beta - \int_{\bdy F} \beta.
\end{equation}
By definition, $\BS(\alpha)$ is the integration over the fiber $C_{2,1}[\Omega]$ of the form 
$\alpha_y \wedge g^*\dVol$. Since $\alpha_y \wedge g^*\dVol$ is closed on $C_{2,1}[\Omega]$, we see $d\BS(\alpha) = \int_{\bdy C_{2,1}[\Omega]} \alpha_y \wedge g^*\dVol$.

Now consider the structure of the boundary of $C_{2,1}[\Omega]$. We are assuming that $x$ is the fixed point, so there are two codimension-one faces of $\bdy C_{2,1}[\Omega]$: one consists of a copy of $\bdy\Omega$ in the form of pairs $(x,y)$ where $y$ is on the boundary; the other is a copy of $S^{2k}$ where $y$ approaches $x$ from some direction.  We note that the outward normal to the fiber points into this $S^{2k}$.

On the $\bdy\Omega$ face, $\alpha_y$ vanishes so there is no contribution to the integral. We now consider the term $\int_{S^{2k}} \alpha_y \wedge g^*\dVol$. What is this form? 

In the definition of integration over the fiber (see Appendix~\ref{gv}), we see that to integrate a $(3k+1)$-form over a $2k$-dimensional fiber and get a resulting $(k+1)$-form, we must write each tangent space to the total space of the bundle as a product of the $2k$-dimensional tangent space to the fiber and the tangent space to the base and decompose our $(3k+1)$-form locally into a wedge of forms on each of these spaces. The fiber portion of the form is then integrated, while the base portion remains.

On $C_{2}[\Omega]$ we now establish the coordinates $x_i$, $z_i = y_i - x_i$, and write $z = r u$, where $u$ is a unit vector. In the bundle \eqref{bsbundle}, the base directions are the $\bdy / \bdy x_i$ and the fiber directions are the $\bdy / \bdy z_i$.

How do these coordinates extend to the boundary of the fiber?  There is no difficulty in defining these coordinates on the boundary face where $y \in \bdy\Omega$. But on the boundary face (12) where the two configuration points coalesce, i.e., where $r=0$, the situation requires a bit more care.  Unlike the standard polar coordinates, in which the $u_i$ will have no meaning when $r=0$, our compactification of the configuration space ensures that the $S^{2k}$ defined by the $u$ coordinates will still be present when $r=0$. 

We now consider the forms $\alpha_y$ and $g^*\dVol$ on the boundary face where $r=0$ with an eye toward integration over the fiber. The form $\alpha_y$ is written entirely in terms of elementary forms chosen from the $dy_i$. But $dy_i = dx_i + dz_i$. And $dz_i = u_i dr + r du_i$, so on this face $\alpha_y$ is written entirely in terms of $dr$ and the $dx_i$. In fact, $\alpha_y$ contains a precise copy of $\alpha_x$ together with a collection of other terms involving $dr$. When we pull this form back to the boundary, the $dr$ terms vanish, leaving only a copy of~$\alpha_x$. On the other hand, the form $g^*\dVol$ is exactly the volume form on the boundary~$S^{2k}$, as the Gauss map in these coordinates is just $g(x,r,u) = u$.  Integrating over the fiber, we obtain
\begin{equation*}
- \int_{\bdy C_{2,1}[\Omega]} \alpha_y \wedge g^*\dVol = \int_{S^{2k}} \alpha_y \wedge g^*\dVol = (\Vol{S^{2k}}) \alpha_x .
\end{equation*}     
Since the standard $2k$-sphere has the opposite orientation of $\bdy C_{2,1}[\Omega]$, the leading minus sign (from \eqref{stokes}) does not appear after the first equality.
\end{proof}

Inspired by the theory of self-adjoint curl operators in dimension $3$ (a long story, stretching from \cite{MR1055988} to \cite{hiptmair-2008}), we observe that 
\begin{lemma}
\label{self adjoint}
$\BS$ is a self-adjoint operator on closed Dirichlet $(k+1)$-forms on $\Omega$, for odd $k$. For any two such forms $\alpha$ and $\beta$, 
\begin{equation*}
\int_{\Omega} \alpha \wedge \BS(\beta) = (-1)^{k+1} \int_{\Omega} \BS(\alpha) \wedge \beta.
\end{equation*}
\end{lemma}

\begin{proof}
We observe that 
\begin{equation*}
d(\BS(\alpha) \wedge \BS(\beta)) = \alpha \wedge \BS(\beta) + (-1)^k \BS(\alpha) \wedge \beta.
\end{equation*}
Integrating both sides over $\Omega$, we see that we must prove that $\int_{\bdy\Omega} \BS(\alpha) \wedge \BS(\beta) = 0$. Since $\BS(\alpha)$ and $\BS(\beta)$ are closed forms on the boundary, this integral depends only on the cup product of the cohomology classes represented by $\BS(\alpha)$ and $\BS(\beta)$ in $H^k(\bdy\Omega)$. 

Borrowing from Theorem~\ref{alexander basis existence} of the Appendix, we know that the de Rham cohomology group $H^k(\bdy\Omega)$ splits into two subspaces: forms with no circulation around $k$-cycles which bound outside $\Omega$ and forms with no circulation around $k$-cycles which bound inside $\Omega$. Since $\alpha$ and $\beta$ vanish outside $\Omega$, $\BS(\alpha)$ and $\BS(\beta)$ are in the first subspace. Further, Theorem~\ref{alexander basis existence} asserts that the cup product of any two forms in the same subspace vanishes. This shows that $[\BS(\alpha)] \cup [\BS(\beta)] = 0$, as desired.
\end{proof}

\subsection{An equivalent definition of helicity as a potential.}

Motivated by Arnold's approach, can we express helicity as the integral of $\alpha \wedge \beta$ for an arbitrary primitive~$\beta$ of $\alpha$?  Unfortunately not, except in special circumstances (see \cite{MR1770976}), since helicity is not gauge-invariant; we must choose an appropriate primitive.  Below, we show that $\BS(\alpha)$ is an appropriate primitive and that we recover the same helicity as in Definition~\ref{fthelicity}.  

\begin{definition}[Primitive definition of helicity]
\label{arnoldhelicity}
Let $\alpha$ be a closed Dirichlet $(k+1)$-form on a compact domain $\Omega$ in $\R^{2k+1}$. The Hodge decomposition theorem for manifolds with boundary tells us that $\alpha$ is exact.  Then the \emph{``Arnol'd helicity''} of $\alpha$ is given by
\begin{equation}
\Hel(\alpha) = \int_\Omega \alpha \wedge \BS(\alpha).  
\end{equation}
\end{definition}

The following proposition ensures that ``Arnol'd helicity'' is equivalent to our original definition of helicity; thus we will refer to both as \emph{helicity}.  The proof is almost immediate.
\begin{proposition}
\label{prop:potential}
Given any closed Dirichlet $(k+1)$-form $\alpha$ on a domain $\Omega$ in $\R^{2k+1}$, the ``Arnol'd helicity'' (via integrating a specific primitive) 
\begin{equation*}
\Hel(\alpha) = \int_\Omega \alpha \wedge \BS(\alpha)
\end{equation*}
of Definition~\ref{arnoldhelicity} is equal to the helicity (via cohomology classes)
\begin{equation*}
\Hel(\alpha) = \int_{C_2[\Omega]} \alpha_x \wedge \alpha_y \wedge g^*\dVol
\end{equation*}
of Definition~\ref{fthelicity}.
\end{proposition}

\begin{proof}
Using the bundle \eqref{bsbundle} and the properties of integration over the fiber, we see that 
\begin{equation*}
\int_{C_2[\Omega]} \alpha_x \wedge \alpha_y \wedge g^*\dVol = \int_\Omega \alpha_x \wedge \left( \int_{C_{2,1}[\Omega]} \alpha_y \wedge g^*\dVol \right) = \int_\Omega \alpha_x \wedge \BS(\alpha_x).
\end{equation*}
\end{proof}

\section{When is helicity invariant under a diffeomorphism?} \label{sec:invariance}

We have now completed our revision of the standard theory of helicity.  With this in hand, we may now fully and precisely answer the question:  is helicity a diffeomorphism invariant?  The answer is negative, except in certain special cases (for one such case, see Proposition~\ref{homotopic-invariance}).  

In the main result of this section, we explicitly calculate the change in helicity of~$\alpha$ under an arbitrary diffeomorphism of $\Omega$.  Specific cases of this formula reproduce the known invariance results about helicity for domains in $\R^3$ (Theorem~\ref{classicalinvariance} and Proposition~\ref{homotopic-invariance}).  

After describing the topology of domains in $\R^{2k+1}$, we first derive the formula for the case where~$\Omega$ is a solid torus in~$\R^3$ before describing the general result.  Even though helicity is the zero function for subdomains $\Omega \subset \R^{4k+1}$ (i.e., $k$ even; see Proposition~\ref{even-k}), we carry out this computation in general and note the instances in which the parity of $k$ matters.  As a check, we confirm that for the~$k$ even case, helicity is invariant under all diffeomorphisms.

We begin by fixing an orientation-preserving diffeomorphism $f\co\Omega \rightarrow \Omega'$ between domains in~$\R^{2k+1}$. Let $\alpha$ be a closed Dirichlet $(k+1)$-form on $\Omega$.  Then its pullback $\alpha' = \fia$ is a closed Dirichlet $(k+1)$-form on $\Omega'$.  By Proposition~\ref{alphaisexact}, both $\alpha$ and $\fia$ are exact.  Did the helicity of $\alpha$ change under the map $f$?  That is, does $\Hel(\alpha)$ equal $\Hel(\alpha')$?  We compute 
\begin{eqnarray*}
\Hel(\alpha) & = & \int_{\Omega} \alpha \wedge \BS(\alpha), \\
\Hel(\alpha') & = & \int_{\Omega' = f(\Omega)}\fia \wedge \BS \left(\fia \right) \\
	& = & \int_{\Omega} f^*\left( \fia \wedge \BS(\fia) \right) \\
	& = & \int_{\Omega} \alpha \wedge f^*\BS\left(\fia \right).
\end{eqnarray*}

Both terms integrate $\alpha$ wedged with a $k$-form, either $\BS(\alpha)$ or $f^*\BS(\alpha')$.  Both $k$-forms are both primitives for $\alpha$, since the exterior derivative commutes with pullbacks, i.e., $df^*\BS\left(\fia \right) = f^* \left( d\BS\left(\fia \right)\right) = f^*\left(\fia \right) = \alpha$.  However, $\BS$ does not in general commute with pullbacks, and so these two $k$-forms are not necessarily equal.  

So we calculate the difference
\begin{equation}
\label{helicity-change}
	\Hel(\alpha') - \Hel(\alpha) = \int_{\Omega} \alpha \wedge \left( f^*\BS(\alpha') - \BS(\alpha) \right)
\end{equation}

In general terms, given two primitives $\beta$ and $\tbeta$ of $\alpha$, we wish to compute  $\int_{\Omega}   \alpha\wedge (\tbeta - \beta)$. We first observe that the integrand is an exact $(2k+1)$-form. In particular,
\begin{equation*}
	d\left((\tbeta - \beta)  \wedge (\tbeta + \beta) \right) =  2 (-1)^{k^2 + 2k} \alpha \wedge (\tbeta - \beta) = 2 (-1)^k \alpha \wedge (\tbeta - \beta).
\end{equation*}
Upon simplifying this potential $2k$-form, we conclude that
\begin{equation}
	 \left((\tbeta - \beta)  \wedge (\tbeta + \beta) \right) = 
	 \begin{cases} 
	    2 \tbeta \wedge \beta & \text{if $k$ is odd},\\ 
	    \tbeta \wedge \tbeta - \beta \wedge \beta & \text{if $k$ is even}. 
	 \end{cases} 
\end{equation}
Applying Stokes' theorem, we obtain
\begin{eqnarray*}
	\int_{\Omega}   \alpha\wedge(\tbeta - \beta) & = &  \int_{\Omega} { \tfrac{1}{2} (-1)^k d\left((\tbeta - 		\beta)  \wedge (\tbeta + \beta) \right)} \\
	& = &  (-1)^k \int_{\bdy\Omega} { \tfrac{1}{2} \left((\tbeta - \beta)  \wedge (\tbeta + \beta) \right)} \\
	& = & 
	\begin{cases} 
	    \int_{\bdy\Omega} \beta \wedge \tbeta & \text{if $k$ is odd},\\ 
	    \tfrac{1}{2} \int_{\bdy\Omega} \tbeta \wedge \tbeta - \beta \wedge \beta & \text{if $k$ is even}. 
	 \end{cases} 
\end{eqnarray*}
Since both $\beta$ and $\tbeta$ are primitives of $\alpha$, and $\alpha$ is Dirichlet, they both are closed on the boundary. On the $2k$-manifold $\bdy\Omega$, the Hodge decomposition theorem tells us that every closed $k$-form can be written as the sum of an exact $k$-form and a $k$-form which represents a de Rham cohomology class in $H^k(\bdy\Omega)$. So write 
\begin{equation*}
\label{beta-gamma}
	\beta = d\phi + \gamma, \quad \tbeta = d\tphi + \tgamma.  
\end{equation*}
We now use this decomposition to analyze $\beta \wedge \tbeta$ on $\bdy\Omega$. Since $\tbeta$ is closed, 
\begin{equation*}
	d\phi \wedge \tbeta = d(\phi \wedge \tbeta).
\end{equation*}
Stokes' Theorem implies that the integral of an exact form on a boundary is zero; thus, $\int_{\bdy\Omega} d\phi \wedge \tbeta = 0$. Continuing this argument, we see that $\int_{\bdy \Omega} \beta \wedge \tbeta = \int_{\bdy\Omega} \gamma \wedge \tgamma$.
This integral is the cup product of the de Rham cohomology classes represented by $\gamma$ and $\tgamma$ in $H^k(\bdy\Omega)$ evaluated on the top class of $\bdy\Omega$.

Similarly, $\int_{\bdy \Omega} \beta \wedge \beta = \int_{\bdy\Omega} \gamma \wedge \gamma$ and $\int_{\bdy \Omega} \tbeta \wedge \tbeta = \int_{\bdy\Omega} \tgamma \wedge \tgamma$.  

Viewing $\beta= \BS(\alpha)$ and $\tbeta= f^*\BS(\alpha')$, we may now represent the change in helicity (\ref{helicity-change}) in terms of primitives that represent cohomology classes:
\begin{equation}
\label{helicity-change-gamma}
	\Hel(\alpha') - \Hel(\alpha) = 
	 \begin{cases} 
	    \int_{\bdy\Omega} \gamma \wedge \tgamma & \text{if $k$ is odd},\\ 
	    \tfrac{1}{2} \int_{\bdy\Omega} \tgamma \wedge \tgamma - \gamma \wedge \gamma & \text{if $k$ is even}. 
	 \end{cases} 
\end{equation}

\subsection{Background on the homology of domains in $\R^{2k+1}$}

Before proceeding, we list a couple of ``folk theorems'' about the homology and cohomology of domains in $\R^{2k+1}$.  To aid the non-expert reader, we also provide an example in Figure~\ref{fig:alexanderbasis}.  We furnish proofs of these results in Appendix~\ref{hom}. In all of these theorems, we use de Rham cohomology and so take our coefficients in $\R$. In this case, the Universal Coefficient Theorem gives us a natural duality isomorphism betweeen homology and cohomology. For a homology class $s$, we denote the dual cohomology class by $s^*$. We start with an existence theorem for a special basis for the $k$-th homology of $\bdy\Omega$:

\begin{thm-b1}
Let $\Omega$ be a compact domain with smooth boundary in $\R^{2k+1}$ or $S^{2k+1}$ (with $k > 0$) and $\ocomp$ be the complementary domain $\R^{2k+1} - \Omega$ or $S^{2k+1} - \Omega$. Then if we take coefficients in $\R$, $H_k(\bdy\Omega) = H_k(\Omega) \oplus H_k(\ocomp)$. Further, given any basis $\langle s_1, \dots, s_n \rangle$ for $H_k(\Omega)$ there is a corresponding basis $\langle s_1, \dots, s_n, t_1, \dots, t_n \rangle$ for $H_k(\bdy\Omega)$ which we call the \emph{Alexander basis} corresponding to $\langle s_1, \dots, s_n \rangle$ so that:
\begin{enumerate}
\item The inclusion $\bdy\Omega \hookrightarrow \Omega$ maps $\langle s_1, \dots, s_n \rangle \in H_k(\bdy\Omega)$ to the original basis $\langle s_1, \dots, s_n \rangle$ for $H_k(\Omega)$ and the inclusion $\bdy\Omega \hookrightarrow \ocomp$ maps $\langle t_1, \dots, t_n \rangle$ to a basis for $H_k(\ocomp)$.
\item 
$s_i = \bdy \sigma_i$ for $\sigma_i \in H_{k+1}(\ocomp,\bdy\ocomp)$, where the $\sigma_i$ form a basis for $H_{k+1}(\ocomp,\bdy\ocomp)$. Similarly, $t_i = \bdy \tau_i$ for $\tau_i \in H_{k+1}(\Omega,\bdy\Omega)$, where the $\tau_i$ form a basis for $H_{k+1}(\Omega,\bdy\Omega)$.
\item 
The cup product algebras of $\Omega$, $\ocomp$ and $\bdy\Omega$ obey 
\begin{equation*}
s_i^* \cup \tau_j^* = \delta_{ij} [\Omega]^*, \quad t_i^* \cup \sigma_j^* = (-1)^{k+1} \delta_{ij}[\ocomp]^*
\end{equation*} and
\begin{equation*}
s_i^* \cup s_j^* = 0, \quad t_i^* \cup s_j^* = \delta_{ij} [\bdy\Omega]^*, \quad t_i^* \cup t_j^* = 0.
\end{equation*}
\item 
The linking number $\Lk(s_i,t_j) = \delta_{ij}$. (Thus $\Lk(t_j,s_i) = (-1)^{(k+1)^2} \delta_{ij}$.) 
\end{enumerate}
The Alexander duality isomorphism from $H_k(\Omega)$ to $H_k(\ocomp)$ maps $s_i$ to $t_i$.  
\end{thm-b1}

We will then study the effect of a homeomorphism on the Alexander basis, proving

\begin{thm-b3}
Suppose that $\Omega$ and $\Omega'$ are compact domains with smooth boundary in $\R^{2k+1}$ or $S^{2k+1}$ and that $f\co \Omega \rightarrow \Omega'$ is an orientation-preserving homeomorphism. Then if $\langle s_1, \dots, s_n \rangle$ is a basis for $H_k(\Omega)$ and $\langle s_1', \dots, s_n' \rangle$ is a corresponding basis for $H_k(\Omega')$ so that $f_*(s_i) = s_i'$, then we may build Alexander bases $\langle s_1,\dots,s_n,t_1,\dots,t_n \rangle$ for $H_k(\bdy\Omega)$ and $\langle s_1',\dots,s_n',t_1',\dots,t_n' \rangle$ for $H_k(\bdy\Omega')$. For these bases, we have
$f_*(\tau_i) = \tau_i'$ and $\bdy f_*(t_i) = t_i'$ so that the map $\bdy f_*\co H_k(\bdy\Omega) \rightarrow H_k(\bdy\Omega')$ can be written as the $2n \times 2n$ matrix
\renewcommand{\arraystretch}{1.5}
\begin{equation}
\bdy f_* = \left[ 
	\begin{array}{c|c}
	I & 0 \\
	\hline 
	(c_{ij}) & I
	\end{array}
	\right], 
\tag{\ref{mmatrix}}
\end{equation}
where each block represents an $(n \times n)$ matrix. If $k$ is odd, the block matrix $\left(c_{ij}\right)$ is symmetric, while if $k$ is even, the block matrix $\left(c_{ij}\right)$ is skew-symmetric.
\end{thm-b3}

Since these theorems are somewhat complicated, we give an example in Figure~\ref{fig:alexanderbasis}.

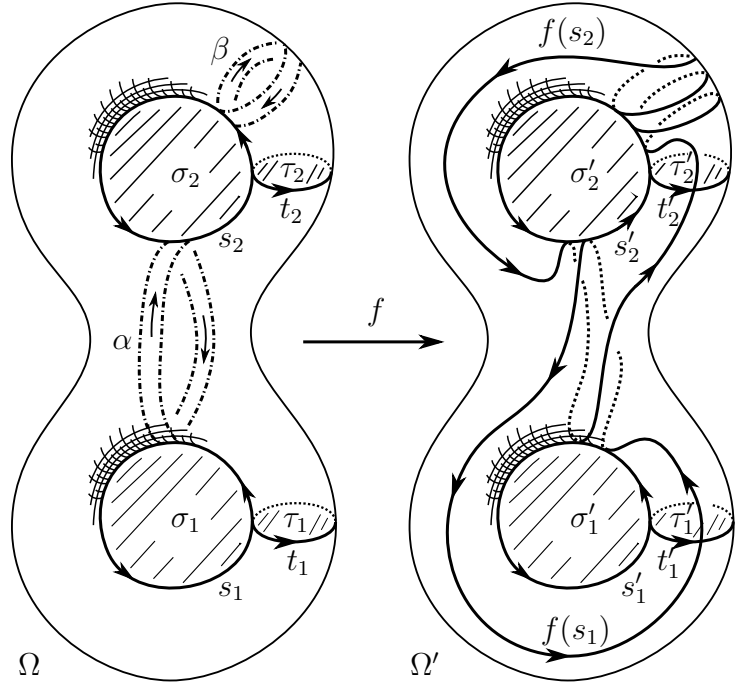
\begin{figure}[ht]
\begin{overpic}[width=3.8in]{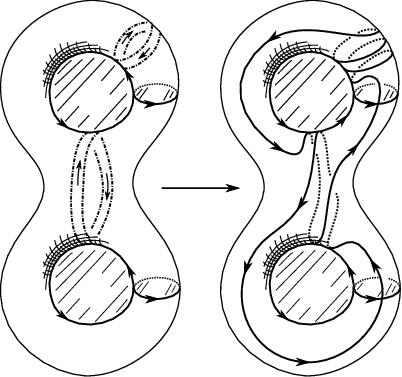}
\put(22,21.5){$\sigma_1$}
\put(37.5,21.5){$\tau_1$}
\put(28.5,12){$s_1$}
\put(38,15.5){$t_1$}
\put(22,69){$\sigma_2$}
\put(28.5,60){$s_2$}
\put(37,69.7){$\tau_2$}
\put(37,64){$t_2$}
\put(1,1){$\Omega$}
\put(55,1){$\Omega'$}
\put(49,50){$f$}
\put(14,46){$\alpha$}
\put(27.5,86){$\beta$}
\put(77,21.5){$\sigma_1'$}
\put(90.7,21.5){$\tau_1'$}
\put(84,12){$s_1'$}
\put(89,15.5){$t_1'$}
\put(77,69){$\sigma_2'$}
\put(90.7,70){$\tau_2'$}
\put(83,59){$s_2'$}
\put(89,64){$t_2'$}
\put(72.5,88.5){$f(s_2)$}
\put(73,6){$f(s_1)$}

\end{overpic}
\caption{Theorem~\ref{alexander basis existence} claims that there exists a special basis $s_1,s_2,t_1,t_2$ for $H_1(\bdy\Omega)$. The left figure shows representatives for these classes. The $s_i$ bound spanning surfaces $\sigma_i \in H_2(\ocomp,\bdy\ocomp)$ in the complement $\ocomp$ of $\Omega$ while the $t_i$ bound surfaces $\tau_i \in H_2(\Omega,\bdy\Omega)$. The map $f$ is the homeomorphism from $\Omega$ to $\Omega'$ given by one Dehn twist around a disk spanning the band $\alpha$ and three Dehn twists around a disk spanning $\beta$. We can see the Alexander basis $s_1',s_2',t_1',t_2'$ on $\Omega$ and the images of $s_1$ and $s_2$. (The images of $t_1$ and $t_2$ are $t_1'$ and $t_2'$.) Further, if $\bdy f_* \co H_1(\bdy\Omega) \rightarrow H_1(\bdy\Omega')$, then we can see $\bdy f_*(s_1) = s_1' - t_1' + t_2'$, $\bdy f_*(s_2) = s_2' + t_1' - 4 t_2'$. This supports the claim of Theorem~\ref{alexander basis transformations} that if we write $\bdy f_*$ as a matrix, it has the special form of~\eqref{mmatrix}, and in particular that the $t_2'$ coefficient $c_{21}$ of $\bdy f_*(s_1)$ is equal to the $t_1'$ coefficient $c_{12}$ of $\bdy f_*(s_2)$.} 
\label{fig:alexanderbasis}
\end{figure}

\subsection{Fluxless case.} \label{sec:fluxless}

\begin{definition}
A closed Dirichlet $(k+1)$-form $\alpha$ on a domain $\Omega$ in $\R^{2k+1}$ is called \emph{fluxless} if the integral $\int_S \alpha = 0$ over every oriented $(k+1)$-cycle $S \subset \Omega$ with $\bdy S \subset \bdy \Omega$.
\end{definition}
We note that since the $(k+1)$-form $\alpha$ represents a de Rham cohomology class in the relative $k$-homology of $\Omega$, the integral $\int_S \alpha$ depends only on the homology class represented by $S$ in $H_{k+1}(\Omega,\bdy\Omega)$. Since $H_{k+1}(\Omega,\bdy\Omega) = H_{k}(\Omega)$ by 
Poincar\'{e} duality, if $\Omega$ has no $k$-homology then every $(k+1)$-form $\alpha$ is fluxless.

Let $\alpha$ be fluxless.  We will utilize some facts about the cohomology and homology of the boundary of a domain in $\R^{2k+1}$. 
See Appendix~\ref{hom} for details.  First,
\begin{equation}
\label{pdecomp}
	H^k(\bdy\Omega) = H^k(\Omega) \oplus H^k(\R^{2k+1} - \Omega).
\end{equation}
We claim that $\gamma$ and $\tgamma$ represent classes entirely in $H^k(\Omega)$. Suppose we have a $k$-cycle  $c$ in $\bdy\Omega$ which represents a class in $H_k(\R^{2k+1} - \Omega)$.  Such a cycle bounds in $\Omega$. Since $\gamma$ and $\beta$ differ by $d\phi$, they have the same integral over $c$. Further, by Stokes' Theorem, the integral of $\beta$ over $c$ is equal to the integral of $\alpha$ on the $(k+1)$-cycle bounded by $c$ in $\Omega$. Since $\alpha$ is fluxless, this integral is zero. Thus $\int_c \gamma = 0$ for every $k$-cycle in $\bdy\Omega$ which represents in $\R^{2k+1}-\Omega$, and (in terms of \eqref{pdecomp}), $\gamma \in H^k(\Omega)$. The same argument shows that $\tgamma  \in H^k(\Omega)$. 

However, in the cup product algebra of $H^k(\bdy\Omega)$, the only pairs of $k$-forms with nontrivial cup products have one member in $H^k(\Omega)$ and one in $H^k(\R^{2k+1}-\Omega)$.  Hence, $\int_{\partial\Omega} \gamma \wedge \tgamma= 0$; likewise, $\int_{\bdy \Omega} \gamma \wedge \gamma = \int_{\bdy\Omega} \tgamma \wedge \tgamma = 0$.  Thus, both cases of \eqref{helicity-change-gamma} are zero, so we have proven
 
\begin{proposition} 
\label{fluxless-invariance}
If $f \co \Omega \rightarrow \Omega'$ is a diffeomorphism between compact domains in $\R^{2k+1}$ with smooth boundary, then for any fluxless $(k+1)$-form $\alpha$ on $\Omega$, its helicity is invariant under $f$, i.e., 
\begin{equation*}
\Hel(\alpha) = \Hel\left( \fia \right) .
\end{equation*}
\end{proposition}

We note that for fluxless forms, it is not necessary to use the Biot-Savart operator in order to define helicity; replacing it with any primitive of $\alpha$ will produce an integral equivalent to helicity (see Definition~\ref{arnoldhelicity}).

But what about closed Dirichlet $(k+1)$-forms $\alpha$ which are not fluxless? To understand the effect of a diffeomorphism on their helicity, we will have to compute the right hand side of~\eqref{helicity-change} directly.  We do so first for a solid torus before proceeding in general.

\subsection{Solid torus example.}
We start with $\Omega$, a solid torus in $\R^3$.  Let $f\co\Omega \rightarrow \Omega'$ be a diffeomorphism, homotopic to $j$ Dehn twists on a spanning disk of $\Omega$. Then $f$ induces isomorphisms of $H_*(\Omega)$ and  $H_*(\bdy\Omega)$.  

By \eqref{pdecomp}, the boundary homology decomposes as $H_1(\bdy\Omega) = H_1(\Omega) \oplus H_1(\R^{3} - \Omega)$.  We choose an \emph{Alexander basis} (defined in Theorem~\ref{alexander basis existence}) $\langle s, t \rangle$:  $t$ is a meridian on $\bdy\Omega$, i.e., $t$ generates $H_1(\R^3 - \Omega)$; $s$ is a longitude on $\bdy\Omega$, i.e., $s$ generates $H_1(\Omega)$.  Choose $s', t'$ similarly on $\bdy\Omega'$.  Let $\sigma$ be a surface in $\R^3-\Omega$ bounded by $s$; let $\tau$ be a surface in $\Omega$ bounded by $t$; similarly define $\sigma'$ and $\tau'$.  Since $f$ applies $j$ Dehn twists to the solid torus $\Omega$, we have $f_*(s) = s' + jt'$ and $f_*(t) =t'$.
 
We consider a closed Dirichlet 2-form $\alpha$ on $\Omega$.  Following the argument above, we utilize the 1-forms $\beta = \BS(\alpha)$ and $\tbeta = f^*\left(\BS(\alpha')\right)$, both primitives for $\alpha$.  Choose suitable 1-forms as above, $\gamma$ and $\tgamma$, which represent in terms of 1-cohomology classes.  From (\ref{helicity-change-gamma}), the change in helicity under $f$ is  
\begin{eqnarray*}
	\Hel(\alpha') - \Hel(\alpha) & = &  \int_{\bdy \Omega} \gamma \wedge \tgamma.
\end{eqnarray*}

We recognize this integral as a cup product pairing in $H^1(\bdy\Omega)$, since both forms in the integrand can be viewed as 1-cohomology classes.  The cup product pairing is straightforward on the torus. If we write the cohomology classes dual to $s$ and $t$ as $s^*$ and $t^*$, then by Theorem~\ref{alexander basis existence}
\begin{equation*}
s^* \cup s^* = 0, \qquad t^* \cup t^* = 0, \qquad t^* \cup s^* =  [\bdy\Omega]^*,
\end{equation*}
where $[\bdy\Omega]$ is the top class of the boundary in $H_2(\bdy\Omega)$ and $[\bdy\Omega]^*$ its dual in 
$H^2(\bdy\Omega)$.  Now we write $\gamma$ and $\tgamma$ in terms of the cohomology classes they represent,
\begin{equation}
[\gamma] = a s^* + b t^* \qquad [ \tgamma ] = \tilde{a} s^* + \tilde{b} t^*, 
\end{equation}
and find the coefficients by integrating.  For example,
\begin{align*}
\tilde{a} &=   \int_{s} \tgamma = \int_{s} f^*\BS\left(\fia \right) - \int_{s} d\tphi 
	 = \int_{f(s)} \BS\left(\fia \right) - 0 \\
	&=  \int_{s'} \BS(\alpha') + j \int_{t'} \BS(\alpha') = \int_{\sigma'} \alpha' + j \int_{\tau'} \alpha'\\
	&= \int_{\sigma'} \alpha' + j \Flux(\alpha',\tau'). 
\end{align*}
Since $\alpha'$ is identically zero on $\R^3-\Omega$, the first term $\int_{\sigma'} \alpha' = 0$.  We also note that $\Flux(\alpha,\tau) = \Flux(\alpha',\tau')$.  Thus, $\tilde{a} = j \Flux(\alpha)$.  By similar computations, we obtain
\begin{equation*}
[ \gamma ] =    \Flux(\alpha) t^*, \qquad 
[ \tgamma ] =   \Flux(\alpha) t^* + j \Flux(\alpha) s^*. 
\end{equation*}
Thus, we can view $\int_{\bdy \Omega'} \gamma \wedge \tgamma$ as a cup product evaluated by integration on the top class of $\bdy\Omega$:
\begin{equation}
\Flux(\alpha) \, t^* \cup \left( \Flux(\alpha) \, t^* +  j \Flux(\alpha) \, s^* \right) = j \Flux(\alpha)^2 \, [\bdy\Omega]^*. 
\end{equation}

In summary, we have proven the following theorem.  

\begin{theorem}
\label{torus-h}
Let $\Omega$ be a solid torus in $\R^3$.  Let $f\co\Omega \rightarrow \Omega'$ be an orientation-preserving map which takes $\Omega$ diffeomorphically to a subset of $\R^3$ and is homotopic to applying $j$ Dehn twists to $\Omega$.  Given a closed Dirichlet 2-form $\alpha$, the change in the helicity of $\alpha$ under $f$ is
\begin{equation*}
\Hel(\fia) - \Hel(\alpha) = j \cdot \Flux(\alpha)^2,
\end{equation*}
where the flux is measured over a spanning surface in $\Omega$ which generates $H_2(\Omega,\bdy \Omega)$.
\end{theorem}

This theorem lets us classify the helicity-preserving diffeomorphisms on the solid torus. We know from Proposition~\ref{homotopic-invariance} that a map from the solid torus to itself preserves helicity for all $2$-forms if it is homotopic to the identity through diffeomorphisms. This theorem lets us prove an (almost) converse result:

\begin{corollary}
\label{torus-g}
If $f\co\Omega \rightarrow \Omega$ is a diffeomorphism of the solid torus to itself, then~$f$ preserves helicity for all closed Dirichlet $2$-forms $\alpha$ if and only if~$f$ is homotopic to the identity through homeomorphisms.
\end{corollary}

\begin{proof}
Wainryb~\cite[Theorem 14]{wain} showed that the mapping class group of a solid torus is isomorphic to   $\Z \oplus \Z_2$, where the $\Z$ counts Dehn twists and the $\Z_2$ detects change of orientation. Thus a map is homotopic to the identity through homeomorphisms if and only if it preserves orientation and has no Dehn twists. By Theorem~\ref{torus-h}, such a map preserves helicity for any form $\alpha$. 

On the other hand, a map which reverses orientation reverses the sign of helicity (for any form $\alpha$ with nonzero helicity) and by Theorem~\ref{torus-h} a map which is homotopic to a nonzero number of Dehn twists changes the helicity of any form $\alpha$ with nonzero flux.
\end{proof}

We now check this theorem with an explicit example. Suppose that $\Omega$ is the solid torus of revolution in $\R^3$ whose core circle has radius 1 and whose tube has radius $R$. We set up (standard) toroidal coordinates $(r, \theta, \phi)$ on the torus where $\theta$ parametrizes the core circle and $(r,\phi)$ are polar coordinates on the cross-sections of the tube. Consider the diffeomorphism $f(r,\theta,\phi) = (r,\theta,\theta+\phi)$ on $\Omega$.  This is a volume-preserving diffeomorphism which applies one Dehn twist to $\Omega$. We will compute the helicity of the 2-form $\alpha = *d\theta$ ($*$ is the Hodge star with respect to the standard form $\dVol_\Omega$) dual to the vector field $\dd{\theta}$ on $\Omega$ before and after the diffeomorphism $f$, as shown in Figure~\ref{fig:straight and twisted}.
\begin{figure}
\begin{overpic}{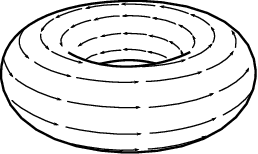}
\end{overpic}
\hspace{0.5in}
\begin{overpic}{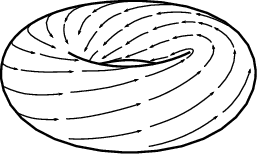}
\end{overpic}
\caption{The figure shows the effect of a diffeomorphism $f$ which applies a Dehn twist to the solid torus $\Omega$ on a vector field $V$ dual to a 2-form $\alpha$. In toroidal coordinates $(r, \theta, \phi)$, this is the map $(r,\theta,\phi) \mapsto (r,\theta,\theta+\phi)$.  The left-hand field $V = \dd{\theta}$.  The field on the right is the pushforward $\dd{\theta} + \dd{\phi}$ of the field $V$ under the map $f$. If the radius of the core circle of the torus is 1 and the radius of the tube is $R$, we compute that the helicity of the left hand field is $0$ and the helicity of the right hand field is $\pi^2 R^4$. This agrees with Theorem~\ref{torus-h}.
}
\label{fig:straight and twisted}
\end{figure}
It is easy to see that $f_* \dd{\theta} = \dd{\theta} + \dd{\phi}$. So we must compute the helicity of these two fields. It is convenient to do this via the ergodic definition of helicity given by Arnol'd~\cite[p.146]{MR1612569}:
\begin{definition}
The \emph{asymptotic linking number} of the pair of trajectories $g^t x_1$ and $g^t x_2$ ($x_1, x_2 \in \Omega)$ of a field $V$ is defined to be the limit
\begin{equation}
\lambda_V(x_1,x_2) := \lim_{t_1, t_2 \rightarrow \infty} \frac{ \lk_V(x_1,x_2;t_1,t_2) }{t_1 t_2}
\end{equation} 
where $\lk_V(x_1,x_2;t_1,t_2)$ is the linking number of the closures of the trajectories extending from $x_1$ and $x_2$ for times $t_1$ and $t_2$. 
\end{definition}
The definition of the asymptotic linking number requires that the trajectories be closed by a ``system of short paths'' joining any given pair of points on $\Omega$ and obeying certain mild technical hypotheses. Luckily, in the cases of interest to us, all of the orbits of our fields are closed with period $2\pi$, so we can ignore these details and let 
\begin{equation}
\lambda_V(x_1,x_2) = \lim_{p, q \in \mathbb{N} \rightarrow \infty} \frac{ \lk_V(x_1,x_2;2\pi p, 2\pi q) }{4 \pi^2 pq}.
\label{eq:asymp}
\end{equation}

Now Arnol'd's ergodic definition of helicity proves that 
\begin{theorem}[Arnol'd~\cite{MR891881,MR1612569}]
The average asymptotic linking number of a divergence-free field tangent to the boundary of a closed domain in $\R^3$ is equal to the helicity of the field. That is, 
\begin{equation}
\Hel(V) = \iint_{\Omega \cross \Omega} \lambda_V(x_1,x_2) \, \dVol_{x_1} \dVol_{x_2}.
\end{equation}
\end{theorem}
We can now compute the helicity of our fields. For the field $\dd{\theta}$, the orbits are all circles parallel to the $xy$ plane. These never link, so the helicity of this field is zero. 

For the field $V = \dd{\theta} + \dd{\phi}$, the orbits are all $(1,1)$ curves on a family of nested tori foliating the solid torus $\Omega$. Any pair of such curves has linking number $1$. Now the trajectories for times $2\pi p$ and $2 \pi q$ cover one of these curves $p$ times and the other $q$ times, so the linking number of the trajectories is $pq$. Taking the limit in~\eqref{eq:asymp}, we get $\lambda_V(x_1,x_2) = 1/4 \pi^2$ for all $x_1$, $x_2$ in $\Omega$. We now compute the helicity of the field 
\begin{equation}
\Hel(V) = \iint_{\Omega \cross \Omega} \lambda_V(x_1,x_2) \dVol_{x_1} \dVol_{x_2} = \frac{\Vol(\Omega)^2}{4\pi^2} = \frac{((2\pi)(\pi R^2))^2}{4\pi^2} = \pi^2 R^4.
\label{eq:direct}
\end{equation}
Here the volume of the tube is the product of the length of the core curve and the cross-sectional area by the Tube Formula.

Now we compare the prediction of Theorem~\ref{torus-h} that $\Hel(V) = \Flux(V)^2$. We must compute the flux of $V$ across a cross-sectional disk of $\Omega$. Since $\dd{\phi}$ is tangent to such a disk, the flux is the same as the flux of $\dd{\theta}$.  A computation shows that this flux is $\pi R^2$, but this is not hard to see: the flux is the rate at which the disk sweeps out volume when rotated around the axis. Since this rate is constant and the disk sweeps out the entire volume $2\pi^2 R^2$ of the tube after rotation through $2\pi$, the rate must be $\pi R^2$, as claimed. We conclude that $\Hel(V) = (\pi R^2)^2$, which agrees with our computation in~\eqref{eq:direct}. We note that a similar example would be easy to work out for a different number of Dehn twists, as a pair of closed orbits of the field after $j$ Dehn twists would have linking number $j$.

\subsection{General formula for change of helicity}

With Theorem~\ref{torus-h} in hand, we now show that a strikingly similar formula holds for general domains $\Omega^{2k+1}$.  We begin with the same setup and compute the change of helicity via (\ref{helicity-change-gamma}). Throughout this section, we will make use of the Einstein summation convention.

Again, we choose an \emph{Alexander basis} (Theorem~\ref{alexander basis existence}) $\langle s_1, \dots, s_n, t_1, \dots, t_n \rangle$ for  $H_k(\bdy\Omega)$. With respect to this basis, we recall that Theorem~\ref{alexander basis transformations} tells us that there is a corresponding Alexander basis $\langle s_1',\dots,s_n', t_1', \dots, t_n' \rangle$ for $H_k(\bdy\Omega')$ so that the map $\bdy f_*\co H_k(\bdy\Omega) \rightarrow H_k(\bdy\Omega)$ looks like a $2n \times 2n$ block matrix
\begin{equation*}
\bdy f_* = \left[ 
	\begin{array}{c|c}
	I & 0 \\
	\hline 
	(c_{ij}) & I
	\end{array}
	\right].
\end{equation*}
We now write the classes $[\gamma]$ and $[\tgamma]$ in terms of this basis.

\begin{proposition} In terms of the Alexander basis, the cohomology classes represented by the forms $\gamma$ and $\tgamma$ are
\label{coeffs}
\begin{equation*}
[\gamma ]  = \Flux(\alpha,\tau_i) t_i^*, \quad
[\tgamma ] = \Flux(\alpha,\tau_i) t_i^* + c_{ij} \Flux(\alpha,\tau_i) s_j^* 
\end{equation*}
where the $c_{ij}$ come from the expression of $\bdy f_*$ as a matrix above.
\end{proposition}

We obtain the general change in helicity formula.
\begin{theorem}
\label{general-change}
Let $\Omega^{2k+1}$ be a subdomain of $\R^{2k+1}$, and let $f\co\Omega \rightarrow \Omega'$ be an orientation-preserving diffeomorphism.  Consider a closed Dirichlet $(k+1)$-form $\alpha$ on $\Omega$.  The change in the helicity of $\alpha$ under $f$ is
\begin{equation}
\label{eq:general-change}
\Hel\left(\alpha' \right) - \Hel(\alpha) = \sum_{i,j} c_{ij} \cdot \Flux(\alpha, \tau_i) \Flux(\alpha, \tau_j)
\end{equation}
where the constants $c_{ij}$ arise from the homology isomorphism induced by $f$ on $H_k(\bdy\Omega)$ as above.  The $(2m+2)$-form $\alpha'$ is the `push-forward' of $\alpha$ under $f$; more precisely, $\alpha'=\fia$ is the pullback of $\alpha$ under the inverse diffeomorphism.
\end{theorem}

\begin{corollary}
\label{vf-change}
Let $\Omega^3$ be a subdomain of $\R^3$, $f\co\Omega \rightarrow \Omega'$ be a volume-preserving diffeomorphism, and $V$ be a smooth vector field on $\Omega$.  The change of helicity of $V$ under $f$ is calculated as above, 
\begin{equation}
\label{eq:vf-change}
\Hel \left(f_* V \right) - \Hel(V) = \sum_{i,j} c_{ij} \cdot \Flux(V, \tau_i) \Flux(V, \tau_j).
\end{equation}
\end{corollary}

This corollary is also true in general when considering the wedge-product of $k$-vectors, dual to a $(k+1)$-form, on $\Omega^{2k+1}\subset \R^{2k+1}$ under volume-preserving diffeomorphisms.  

\begin{proof}[Proof of Proposition~\ref{coeffs}]
The coefficients for $[\gamma]$ and $[\tgamma]$ can be directly calculated.  Write $[\gamma ] = a_i  s_i + b_i t_i$ and $[\tgamma ]  = \tilde{a}_i  s_i + \tilde{b}_i t_i$.  Then,
\begin{equation*}
a_i = \int_{s_i} \gamma = \int_{s_i = \bdy\sigma_i} \BS(\alpha )  - \int_{s_i} d\phi \\
= \int_{\sigma_i} d\BS(\alpha) - 0 = \int_{\sigma_i} \alpha
\end{equation*}
by Stokes' Theorem.  But $\alpha \equiv 0$ outside of $\Omega$, where $\sigma_i$ is located.  Thus $a_i = 0$. We can similarly calculate $b_i$
\begin{equation*}
b_i  =  \int_{\tau_i} d\BS(\alpha) - 0 = \int_{\tau_i} \alpha = \Flux(\alpha,\tau_i). 
\end{equation*}
Calculating $\tilde{a_i}$ and $\tilde{b_i}$ is more involved.  
\begin{equation*}
\tilde{a_j} = \int_{s_j} \tgamma = \int_{s_j} f^*(\BS(\fia ))  - \int_{s_j} d\tphi 
 =  \int_{f(s_j)} \BS(\alpha' ) - 0.
\end{equation*}
By (\ref{mmatrix}), the image $f(s_j)$ is homologous to $s_j' + c_{ij} t_i'$. Since $\BS(\fia)$ is a closed form on $\bdy\Omega$, this integral is equal to  
\begin{align*}
\tilde{a_j} &= \int_{s_j'} \BS(\alpha') +  c_{ij} \int_{t_i'} \BS(\alpha' ) \\
            &= \int_{\sigma_j'} d\BS(\alpha' ) +  c_{ij} \int_{\tau_i'} d\BS(\alpha' ) \\
            &= \int_{\sigma_j'} \alpha'  +  c_{ij} \int_{\tau_i'} \alpha'. 
\end{align*}
Now $\alpha'$ vanishes outside of $\Omega'$, which is where $\sigma_j'$ is located, so the first term $\int_{\sigma_j'} \alpha'$ must be zero.  We compute the second integral on the original $\Omega$, 
\begin{equation*}
\int_{\tau_i'} \alpha' = \int_{f^{-1}\left(\tau_i' \right)} \alpha 
\end{equation*}
We know from Theorem~\ref{alexander basis transformations} that $f^{-1}(\tau_i')$ is homologous to $\tau_i$. 
So we conclude
\begin{equation*}
\tilde{a_j} = c_{ij} \int_{f^{-1}\left( \tau_i' \right)} \alpha = c_{ij} \int_{\tau_i} \alpha = c_{ij} \Flux(\alpha,\tau_i).
\end{equation*}



We compute $\tilde{b_i}$ similarly:
\begin{equation*}
\tilde{b_i}  =   \int_{t_i} \tgamma = \int_{t_i} f^*(\BS(\fia )  - \int_{t_j} d\tphi  =   \int_{f(t_i)} \BS(\alpha' ). 
\end{equation*}
Since $f(t_i)$ is homologous to $t_i'$ and $\BS(\alpha')$ is closed on $\bdy\Omega'$, 
\begin{equation*}
\tilde{b_i} =   \int_{t_i'} \BS(\alpha' ) = \int_{\tau_i'} d\BS(\alpha' ) 
            =   \int_{\tau_i'} \alpha' = \int_{\tau_i} \alpha = \Flux(\alpha,\tau_i), 
\end{equation*}
using our previous observation that $f^{-1}(\tau_i')$ is homologous to $\tau_i$. This completes the proof.
\end{proof}

We can now derive the change in helicity formula using these coefficients.  This will depend on the parity of $k$,  so we start with the case where $k$ is odd.  The change in helicity formula (\ref{helicity-change-gamma}) is
\begin{equation*}
([\gamma] \cup [\tgamma])([\bdy\Omega]) = b_j \tilde{a_j} + a_j \tilde{b_j} = c_{ij} \Flux(\alpha,\tau_i)\Flux(\alpha,\tau_j) + 0.
\end{equation*}
For $k$ even, the change in helicity formula (\ref{helicity-change-gamma}) is $\nicefrac{1}{2} \left( [\tgamma] \cup [\tgamma] - [\gamma] \cup [\gamma] \right) ([\bdy\Omega])$.  
\begin{align*}
([\tgamma] \cup [\tgamma])([\bdy\Omega]) &= 2 \tilde{a_j} \tilde{b_j} = 2 c_{ij} \Flux(\alpha,\tau_i) \Flux(\alpha,\tau_j) \\ 
(\left[\gamma \right] \cup [\gamma])([\bdy\Omega]) &= a_i b_i = 0.
\end{align*}

We have calculated the change of helicity to be 
\begin{equation}
\label{boxformula}
\boxed{ \Hel\left(\alpha' \right) - \Hel(\alpha) = \sum_{i,j} c_{ij} \Flux(\alpha,\tau_i)\Flux(\alpha,\tau_j).  }
\end{equation}
which proves Theorem~\ref{general-change}.  By Theorem~\ref{alexander basis transformations}, the matrix $(c_{ij})$ is skew-symmetric for~$k$ even, so the double sum on the right-hand-side of (\ref{boxformula}) vanishes in this case. Of course this is what we expect, since we know that in this case $\Hel(\alpha) = \Hel(\alpha') = 0$ by Proposition~\ref{even-k}.

\subsection{Classifying the helicity-preserving diffeomorphisms}

We can now classify completely the helicity-preserving maps $f\co\Omega \rightarrow \Omega$.  For even $k$,  $\Omega$ is $(4m+1)$-dimensional, helicity is the trivial invariant, and all maps are helicity-preserving.  In the case where $k$ is odd, $\Omega$ is $(4m+3)$-dimensional, and the helicity-preserving maps are the orientation-preserving maps with $c_{ij} \Flux(\alpha,\tau_i)\Flux(\alpha,\tau_j) = 0$ for all $\alpha$.  These maps form a subgroup $\HP$ of the diffeomorphism group of $\Omega$.  In Theorem~\ref{alexander basis transformations} the $c_{ij}$ were determined by the homotopy type of $f$.  Since they surely vanish for the identity map, $\HP$ also forms a subgroup of the smooth mapping class group of $\Omega$. In this language, we can give a more standard description of the group $\HP$.

For a surface, the Torelli subgroup of the mapping class group is the group of homeomorphisms which act trivially on homology~\cite{MR718141}. Analogously, for any manifold $M$ we will call the group of diffeomorphisms which act trivially on homology the Torelli subgroup $\Torelli(M)$ of the smooth mapping class group of $M$. Note that maps in $\Torelli(M)$ are orientation-preserving. It is easy to see
\begin{corollary}
\label{classification}
Since a homeomorphism from $\Omega$ to itself naturally maps $\bdy\Omega$ to itself, there is a natural inclusion $\Torelli(\bdy\Omega) \subset \Torelli(\Omega)$. With respect to this inclusion, if $\Omega$ is a domain in $\R^{2k+1}$ with $k$ odd,
\begin{equation*}
\HP(\Omega) \cap \Torelli(\Omega) = \Torelli(\bdy\Omega) . 
\end{equation*}
\end{corollary}

\begin{proof}
If $f \in \Torelli(\Omega)$, then $f$ is orientation-preserving and Theorem~\ref{general-change} applies. The right-hand side of~\eqref{eq:general-change} is the action of the quadratic form defined by the matrix~$(c_{ij})$ on the vector $\Flux(\alpha,\tau_i)$. Since any vector of fluxes can be obtained by choosing an appropriate $\alpha$, this vanishes for all $\alpha$ if and only if the matrix $(c_{ij})$ is skew-symmetric. But since $k$ is odd, the matrix $(c_{ij})$ is symmetric by Theorem~\ref{alexander basis transformations}, so~$f$ is helicity-preserving if and only if all the $c_{ij}$ are zero. Looking at our construction, we see that if $f \in \Torelli(\Omega)$, then $s_i' = s_i$. In fact, this means $t_i' = t_i$ as well. Thus $f$ acts trivially on $H_*(\bdy\Omega)$ if and only if all the $c_{ij}$ vanish. This completes the proof. 
\end{proof}

We note that $\HP(\Omega)$ is generally somewhat larger than $\Torelli(\bdy\Omega)$: if $\Omega$ is a handlebody with $n$ handles any (orientation-preserving) permutation of the handles will surely preserve helicity, but not act trivially on $H_*(\Omega)$ or $H_*(\bdy\Omega)$. In our construction above, the matrix $M$ will still be the identity matrix due to our careful choice of basis.

We have now given a full account of the interplay between the map $f$ and the form~$\alpha$ in determining the effect of a mapping on the helicity of a form. Our previous theorems are now revealed as easy corollaries of Theorem~\ref{general-change}: Proposition~\ref{fluxless-invariance} states that if $\alpha$ is fluxless, then the helicity of $\alpha$ is preserved by any diffeomorphism $f$. Indeed, the $\Flux(\alpha,\tau_i)$ vanish for all $i$, which means that the right hand side of~\eqref{boxformula} vanishes and helicity is invariant.

On the other hand, Proposition~\ref{homotopic-invariance} states that if the map $f$ is homotopic to the identity map, then helicity is invariant under $f$ for any form. In our new language, this is the trivial statement that the identity element in the mapping class group of~$\Omega$ is in the subgroup $\HP$. We have seen above in Corollary~\ref{torus-g} that $\HP = \{e\}$ for the solid torus. But in general $\HP$ is much larger. Figure~\ref{fig:doubletorus} shows an explicit example of a helicity-preserving diffeomorphism of a domain in $\R^3$ which is not homotopic to the identity. Let $\Omega$ be the solid 2-holed torus, and let $\alpha$ be the curve around the ``waist'' of the torus. Our map $f$ will be a Dehn twist around $\alpha$ extended to the interior of $\Omega$ along the spanning disk  for $\alpha$ shown in the picture.

\begin{figure}
\begin{overpic}{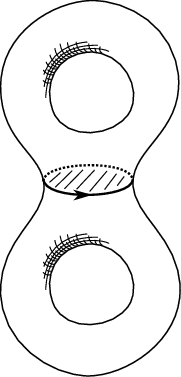}
\put(22,43){$\alpha$}
\end{overpic}
\caption{The figure shows how to construct a diffeomorphism from the solid two-holed torus $\Omega$ to itself which is helicity-preserving for all closed Dirichlet $2$-forms on $\Omega$ but not homotopic to the identity. We take a Dehn twist around the curve $\alpha$ on $\bdy\Omega$ and extend the twist to the interior of $\Omega$ across the spanning disk shown above. The resulting map induces the identity on the homology of $\Omega$ and $\bdy\Omega$, and so preserves helicity by Corollary~\ref{classification}, but it is not homotopic to the identity.}
\label{fig:doubletorus}
\end{figure}

\begin{lemma} The map $f$ of Figure~\ref{fig:doubletorus} is helicity-preserving but not homotopic to the identity map.
\end{lemma}

\begin{proof}
To see the first part, by the corollary above we must show that the map $(\bdy f)_*\co H_1(\bdy\Omega) \rightarrow H_1(\bdy\Omega)$ is the identity. But we can take a set of generators for $H_1(\bdy\Omega)$ which are fixed by $\bdy f$ as long as we stay away from $\alpha$.

To show the second, observe that if $f$ was homotopic to the identity, then $\bdy f$ would be as well. But $\bdy f$ is a Dehn twist around an essential curve in $\bdy\Omega$, so $\bdy f$ is nontrivial in the mapping class group of $\bdy\Omega$~\cite[Proposition 2.1]{fmprimer}.
\end{proof}

\section{Future Directions}

Our perspective on helicity has allowed us to observe three new kinds of invariance for $\Hel(\alpha)$: invariance under change of volume form on $S^{2k}$, invariance in the cohomology class $[\alpha_x \wedge \alpha_y]$ in $H^{2k+2}(C_2[\Omega],\bdy C_2[\Omega])$, and invariance under diffeomorphisms of $\Omega$ which preserve the homology of $\bdy\Omega$. We have made stronger the analogy between helicity for forms and finite-type invariants for knots and links. And we have explained the effect of any diffeomorphism of $\Omega$ on the helicity of a form $\alpha$. We devote the rest of the paper to observing some immediate consequences of our point of view, and to suggesting some future directions for further study.

\subsection{Submanifold helicities}

So far we have only considered the case where $\Omega$ is a top-dimensional subdomain of $\R^{2k+1}$. We can define an analogous helicity just as easily for closed Dirichlet $(k+1)$-forms on an $n$-dimensional submanifold $\Omega$ of $\R^{m}$ by
\begin{equation}
\Hel(\alpha) = \int_{C_2[\Omega]} \alpha_x \wedge \alpha_y \wedge g^* \dVol_{S^{m-1}}
\end{equation}
as long as the integrand $\Phi_m=\alpha_x \wedge \alpha_y \wedge g^* \dVol_{S^{m-1}}$ is a $2n$-form.  This requires that $2k + 2 + m-1 = 2n$, i.e., that $m = 2n - 2k - 1$.  We refer to such an integral as a \emph{$(k,n,m)$-helicity} and note that the helicity from Definition~\ref{fthelicity} is the $(k,2k+1,2k+1)$-helicity. 

\begin{question} \label{questions}
Two questions arise immediately:
\begin{enumerate}
\item \label{qinvt} For which values of $k$, $n$ and $m$ is $(k,n,m)$-helicity an invariant?
\item \label{qtrivial} When is the invariant nontrivial?
\end{enumerate}
\end{question}

As before, we know that $(k,n,m)$-helicity will be an invariant if the closed form $\alpha_x \wedge \alpha_y \wedge g^* \dVol_{S^{m-1}}$ is Dirichlet, i.e., if it vanishes on the boundary of $C_2[\Omega]$.  Following the proof of Lemma~\ref{alpha}, we only have to worry about the face $(12)$ of this boundary, which is diffeomorphic to $\Omega \cross S^{m-1}$. On this face, $\alpha_x \wedge \alpha_y$ pulls back to $\alpha \wedge \alpha$. Our previous argument depended on the observation that this was a $(2k+2)$-form $\alpha \wedge \alpha$ on the $n=(2k+1)$-manifold $\Omega$. In general, $2k + 2$ may not be greater than $n$, so we cannot depend on this argument.  However, we note that if $k+1$ is odd, then $\alpha \wedge \alpha$ vanishes by antisymmetry, providing a partial answer to the first question above.  A standard example here is $(0,2,3)$-helicity, which should measure the linking of a 1-form on a surface in $\R^3$.  We do not know whether the $(1,5,7)$-helicity measuring the linking of a $2$-form on a 5-dimensional surface in $\R^7$ is an invariant. The $(-1,1,3)$-helicity of $0$-forms on a curve in $\R^3$ turns out to be precisely the writhing number of the curve, so we know that this helicity is not an invariant.

What about the second question?  For a contractible domain $\Omega = D^{n}$, the configuration space $C_2[\Omega]$ has the topology of $D^n \cross D^n - \{ \text{pt} \} = D^{n+1} \cross S^{n-1}$ as we saw above. In this case, only the cohomology groups $H^*(C_2[\Omega],\bdy C_2[\Omega])$ where $* = 0, n+1, 2n$ are nontrivial. So for a $(k,n,m)$-helicity to be nontrivial in this case, we must have $2k+2 = n+1$, which only occurs for $(k,2k+1,2k+1)$-helicity. But if $\Omega$ has nontrivial homology, then $C_2[\Omega]$ has more homology and $(k,n,m)$-helicity might be nontrivial.  For example, we conjecture that if $\Omega$ has $1$-dimensional homology, then the invariant $(0,2,3)$-helicity is nontrivial on $\Omega$.      

When $k$ was even, we showed in Proposition~\ref{even-k} that helicity could only extend to a function that was identically zero.  There is a corresponding result for $(k,n,m)$-helicities:

\begin{proposition}
If $k+n$ is even, then the $(k,n,m)$-helicity is identically zero for any $(k+1)$-form $\alpha$.
\end{proposition}

\begin{proof}
The argument is similar to that which proves Proposition~\ref{even-k}.  We consider the automorphism $a$ of $C_2(\Omega)$ that interchanges $x$ and $y$; it extends naturally to $C_2[\Omega]$.  It changes the orientation of $C_2[\Omega]$ by a factor of $(-1)^{n^2}$.  

Next, we take the pullback $a^*\Phi_m = \alpha_y \wedge \alpha_x \wedge a^*g^*\dVol_{S^{m-1}}$. The map $a$ induces an antipodal map on $S^{m-1}$; since $m$ is odd, such a map has degree $-1$.  Hence, $a^*g^*\dVol_{S^{m-1}}= -g^*\dVol_{S^{m-1}}$.  Also, $\alpha_y \wedge \alpha_x = (-1)^{(k+1)^2} \alpha_x \wedge  \alpha_y$.  Combining these results, $a^*\Phi = (-1)^k \Phi$.  We then compute
\begin{align*}
\int_{C_2[\Omega]} a^*\Phi & = \int_{a\left(C_2[\Omega] \right)} \Phi \\
\int_{C_2[\Omega]} (-1)^k \Phi & = \int_{(-1)^n C_2[\Omega]} \Phi \\
(-1)^k \Hel(\alpha) & = (-1)^n\Hel(\alpha)
\end{align*}

If $k$ and $n$ have the opposite parity, this implies that $\Hel(\alpha) = -\Hel(\alpha)$, i.e., that helicity is zero, and proves our proposition.  If $k$ and $n$ have the same parity, the conclusion is a tautology: $\Hel(\alpha) = \Hel(\alpha)$.
\end{proof}

\subsection{Helicity on 3-manifolds with boundary}

Arnol'd and Khesin give a definition for helicity for $2$-forms on a simply-connected $3$-manifold $M$ without boundary in~\cite{MR1612569}. If the manifold is not simply connected, their method works for ``fluxless'' $2$-forms which do not represent nontrivial classes in $H^2(M)$, but fails for forms which do represent in $H^2(M)$. Our work so far allows us to define and understand helicity for all forms on $2k+1$-manifolds $M^{2k+1}$ with boundary with an embedding into $\R^{2k+1}$. As we have shown, the helicity of a form on such a domain depends on the embedding of the domain into $\R^{2k+1}$. However, it is easy to remove this dependence, allowing us to define a kind of helicity on some 3-manifolds with boundary which is independent of their embedding into Euclidean space.

\begin{definition}
Let $M^{2k+1}$ be a compact, oriented manifold with smooth boundary which admits an orientation-preserving diffeomorphic embedding $f$ into $\R^{2k+1}$. Let $\alpha$ be a Dirichlet $(k+1)$-form on $M^{2k+1}$. Let $\tau_1, \dots, \tau_n$ be a basis for $H^{k+1}(M,\bdy M)$, and let $F(\alpha)$ be the largest number so that $\Flux(\alpha,\tau_i) = k_i F(\alpha)$ for some $k_i \in \Z$ or $0$ if no such $F(\alpha)$ exists.
s
We define the \emph{residual helicity} of $\alpha$ to be
\begin{equation*}
\operatorname{ResHel}(\alpha) = \Hel(f_*(\alpha)) \mod F(\alpha).
\end{equation*}
if $F(\alpha) > 0$ and $\operatorname{ResHel}(\alpha) = 0$ otherwise.
\end{definition}

We have immediately from Theorem~\ref{general-change} that

\begin{corollary}
The residual helicity $\operatorname{ResHel}(\alpha)$ does not depend on $f$ and is a diffeomorphism invariant of $\alpha$.
\end{corollary}

On most domains, this helicity is fairly weak, since it can vanish if the fluxes of the form $\alpha$ are irrational multiples of one another. So this definition is really most useful when $H^{k+1}(M) = \R$, as in the solid torus. But it may point the way towards defining a more powerful version of helicity on an arbitrary $2k+1$-manifold with boundary.

\subsection{Cross-helicities}

So far we have only considered configuration spaces of two points in a single domain in $\R^{2k+1}$. But we could also construct similar configuration spaces where the points are restricted to lie in different domains. For instance, consider the configuration space $X \cross Y$ where $X$ and $Y$ are disjoint linked solid tori of the form $S^{k} \cross D^{k+1}$ in $\R^{2k+1}$. This configuration space simply restricts one point to lie in each torus. As before, there is a Gauss map $g\co X \cross Y \rightarrow S^{2k}$ and we can define the cross-helicity of a pair of closed Dirichlet $(k+1)$-forms $\alpha_x$ and $\alpha_y$ defined on $X$ and $Y$ by 
\begin{equation*}
\Hel(\alpha_x,\alpha_y) = \int_{X \cross Y} \alpha_x \wedge \alpha_y \wedge g^* \dVol_{S^{2k}}.
\end{equation*}
As before, we observe that $\alpha_x \wedge \alpha_y$ is a closed Dirichlet $(2k+2)$-form on $X \cross Y$. But this means that $\alpha_x \wedge \alpha_y$ represents a cohomology class in $H^{2k+2}(X \cross Y)$.

Now $\alpha_x$ and $\alpha_y$ represent classes in $H^{k+1}(X,\bdy X) \simeq \R$ and $H^{k+1}(Y,\bdy Y) \simeq \R$. Since $X \simeq Y \simeq S^{k} \cross D^{k+1}$, $H_{k+1}(X,\bdy X)$ and $H_{k+1}(Y,\bdy Y)$ are generated by cycles spanning the $D^{k+1}$ and the classes represented by $\alpha_x$ and $\alpha_y$ are determined by their flux across these spanning cycles. Let us call the cohomology duals to the spanning cycles $[g_x]$ and $[g_y]$ so that 
\begin{equation*}
[\alpha_x] = \Flux(\alpha_x) [g_x] \in H^{k+1}(X,\bdy X), \qquad [\alpha_y] = \Flux(\alpha_y) [g_y] \in H^{k+1}(Y,\bdy Y).
\end{equation*}
We note that $g_x$ and $g_y$ are the Poincar\'{e} duals of the generators $s_x$ and $s_y$ for the homology of the $S^k$ in $X$ and $Y$ with respect to the top classes in $H^{2k+1}(X)$ and $H^{2k+1}(Y)$ which integrate to $1$ on their respective domains.

In $X \cross Y$, the cohomology class represented by $\alpha_x \wedge \alpha_y$ is simply $\Flux(\alpha_x) \Flux(\alpha_y) [g_x] \wedge [g_y]$, which is the Poincar\'{e} dual in $X \cross Y$ of $[s_x] \wedge [s_y]$. Now if we restrict the Gauss map to the core $S^k \cross S^k$ of $X \cross Y$, we see that the pullback of the volume form on $S^{2k}$ represents the cohomology class $\Lk(X,Y) [s_x] \wedge [s_y]$. (Here $\Lk$ is the linking number of $X$ and $Y$ in $\R^{2k+1}$.)

This reproves the standard result that 
\begin{equation*}
\Hel(\alpha_x,\alpha_y) = \Flux(\alpha_x) \Flux(\alpha_y) \Lk(X,Y)
\end{equation*}
using our language for forms in linked tubes.

As in the standard helicity integral, the pullback of the area form on $S^{2k}$ to our configuration space is a multiple of the Poincar\'{e} dual of $\alpha_x \wedge \alpha_y$. In the original helicity integral, where $\alpha_x$ and $\alpha_y$ were pulled back from the same $\alpha$ this multiple measured a new topological property of the form $\alpha$. We could see that we were measuring  new information in this case because the homology of $C_2[\Omega]$ had a new class in $H^{2k}(C_2[\Omega])$ which was not generated by the topology of the domain $\Omega$.

On the other hand, when we calculate the cross-helicity of linked tubes, all of the homology classes involved are generated by the topology of the original domains $X$ and $Y$. This means that the cross-helicity is really an invariant of the core spheres of $X$ and $Y$-- the forms $\alpha_x$ and $\alpha_y$ are multiples of the Poincar\'{e} duals to the generators of these spheres, and contribute no interesting information other than their fluxes.

Similarly, several authors have defined ``triple-helicity'' integrals for the case of three divergence-free vector fields defined on three solid tori $X$, $Y$, and $Z$ in space. The resulting vector field (or form) invariants turn out to be equal to 
\begin{equation*}
\Flux(\alpha_x) \Flux(\alpha_y) \Flux(\alpha_z) I(X,Y,Z)
\end{equation*}
where $I(X,Y,Z)$ is a topological invariant of the three tubes. For example, a theorem of this kind appears in Proposition 6.1 of  Komendarczyk~\cite{MR2544738}.

We can now see that while such theorems are appealing, none of these integrals is likely to easily generalize to a meaningful invariant of 2-forms on a contractible domain in $\R^3$ defined by integration over $C_3[D^3]$. We could repeat the procedure above and generate a closed Dirichlet $6$-form $\alpha_x \wedge \alpha_y \wedge \alpha_z$ on  $C_3[D^3]$. Unfortunately, in the $9$-dimensional space $C_3[D^3]$, the cohomology group $H^{6}(C_3[D^3],\bdy C_3[D^3]) \simeq H_3(C_3[D^3]) \simeq 0$, since the (absolute) cohomology of $C_3[D^3]$ is known to be generated by 2-forms coming from the three Gauss maps $g_{xy}(x,y,z) = \nicefrac{x-y}{\norm{x-y}}$, $g_{yz}(x,y,z) = \nicefrac{y-z}{\norm{y-z}}$ and $g_{zx}(x,y,z) = \nicefrac{z-x}{\norm{z-x}}$. Thus any such triple-helicity integral must be zero. It remains an important open problem to construct a nontrivial triple-helicity integral for forms on a contractible domain.

\subsection{Helicity and finite-type invariants}

On a last and somewhat speculative note, we wonder whether the finite-type invariants (expressed as integrals over certain configuration spaces of points on a knot) could be used to obtain integral invariants for divergence-free fields tangent to the boundary of a single knotted flux tube $\Omega$. The values of invariants would not be interesting-- we expect each to have the value $\Flux(\alpha) I(\Omega)$ where $I(\Omega)$ is the corresponding finite-type invariant of the tube $\Omega$-- but the integrals could in principle be used to obtain sharper energy bounds for such vector fields than the classical results of Freedman and He~\cite{fh1}. The major obstacle here seems to be that the construction of the finite-type invariants as integrals depends on the fact that the configuration spaces of circles are disconnected (the order of points on the circle cannot change in a connected component), allowing different components to be attached to one another to form more complicated spaces. We do not yet understand the analogous constructions for configuration spaces of points in solid tori.

\section{Acknowledgements}

The authors are grateful for many colleagues and friends with whom we have had fascinating and productive conversations. We are particularly indebted to Fred Cohen, Elizabeth Denne, Dennis DeTurck, Herman Gluck, Martha Holmes, Jamie Jorgensen, Will Kazez, Tom Kephart, Rafal Komendarczyk, Robert Kotiuga, Clint McCrory, Paul Melvin, Clay Shonkwiler, Jim Stasheff, and Shea Vela-Vick. 

\bibliography{helicity-forms,cantarella}{}
\bibliographystyle{plain}

\appendix
\section{Stokes' Theorem for Product Manifolds} \label{gv} 

Let $X^n$ and $Y^m$ be manifolds with boundary, where $n$ is finite but $m$ may be infinite.  Consider the product manifold $M = X \times Y$.  Let $\alpha$ be a smooth $(n+k)$-form on $M$, for $k \geq 0$.  By integrating $\alpha$ over $X$, we can construct a map
\begin{equation*}
\begin{array}{lccl}
\pi_*  \co & \Lambda^{n+k}(X \times Y) & \rightarrow  & \Lambda^{k}(Y), \\
\pi_*  \co & \alpha 					& \mapsto	  & \int_X{ \alpha }.
\end{array}
\end{equation*}

Here we are following Volic's notation~\cite{MR2300426} of $\pi_*$, even though this map is not a push-forward of forms; rather we map merely via integration.

\renewcommand{\arraystretch}{1.5}

{\bf Stokes' Theorem.} Via this setup, the differential of the $k$-form $\pi_* \alpha$ on $Y$ is
$$\begin{array}{lcccc}
d\pi_* \alpha   	& = &  \pi_* d\alpha & - & (\partial \pi)_* \alpha \\
d \int_X { \alpha } & = &  \int_X { d\alpha } & - & \int_{\partial X} {\alpha} 
\end{array}$$

{\it Rationale.}  Express $\alpha$ as the sum of three smooth forms:  $\alpha = \alpha_n + \alpha_{n-1} + \beta$, where $\alpha_n = \dVol_x \wedge \cdots$ includes $n$ elementary $dx_i$ forms, $\alpha_{n-1}$ includes $n-1$ elementary $dx_i$ forms, and $\beta$ has less than $n-1$ elementary $dx_i$ forms.  We consider these three forms in Stokes' Theorem above.  Sample terms include:
$$\begin{array}{|c|lccc|}
\hline
\mathrm{Form} && \mathrm{Sample\ Term} && \\
\hline
\alpha_n 	& f(x,y) \; & \dVol_x 	& \; \wedge \; & dy_{i_1} \wedge \cdots \wedge dy_{i_k} \\
\alpha_{n-1} & f(x,y) \; & dx_1 \wedge \cdots \wedge \widehat{dx_j} \wedge dx_n & \; \wedge \; & dy_{i_1} \wedge \cdots \wedge dy_{i_{k+1}} \\
\beta & f(x,y) \; & dx_1 \wedge \cdots \wedge \widehat{dx_j} \cdots \widehat{dx_k}  \wedge dx_n & \; \wedge \; & dy_{i_1} \wedge \cdots \wedge dy_{i_{k+2}} \\
\hline
\end{array}$$
\vspace{0.05in}
\noindent Here the hat on $\widehat{dx_j}$ reports that term does not appear in the wedge product.

Consider the form $\alpha_n$ first.  Note that $\int_{\partial X} {\alpha_n} = 0$ since $\alpha_n$ has greater dimension in $x$ variables than the dimension of $\partial X$.  Since $\alpha_n$ already contains the volume form on $X$, its differential $d\alpha_n$ will only introduce $y$ terms, so we may differentiate under the integral sign:
\begin{equation*}
d \int_X { \alpha_n }  =   \int_X { d\alpha_n } \qquad. 
\end{equation*}

The form $\alpha_{n-1}$ does not contain the entire volume form on $X$, so $ \int_X { \alpha_{n-1} } = 0$.  When computing $\int_{\partial X} {\alpha_{n-1}(x,y)}$, we may hold $y$ fixed and apply Stokes' Theorem on $X$ to each elementary form of $(n-1)$ $dx$ terms; thus $\int_{\partial X} {\alpha_{n-1}} = \int_X { d\alpha_{n-1} }$.

Finally, the third term $\beta$ is smooth and does not yield a top-dimensional form (in terms of $x$ variables) in any of the three integrals, so they all vanish:
$$d \int_X { \beta }  =   \int_X { d\beta } = \int_{\partial X} {\beta} = 0 \qquad.$$ 

For our purposes usually $k=0$, and it is difficult to calculate $d\int_X \alpha$ directly.  Thus, we shall invoke Stokes' Theorem.  Usually, either $\alpha$ is closed or the integral of $d\alpha$ is straightforward, so our efforts concentrate upon computing $\int_{\partial X} {\alpha}$.  Its important terms have $(n-1)$ $dx$ terms and include a $dy$ term.  When $Y$ is the space of embeddings or knots, we may view this term as dual to a variational vector field.

\section{The cohomology of domains in $\R^{2k+1}$ (or $S^{2k+1}$)} 
\label{hom}  

The purpose of this section is to give a self-contained exposition of some basic facts about the homology of domains in $\R^{2k+1}$ (or $S^{2k+1}$). In dimension 3, we gave an exposition of much of this material (without detailed proofs) in~\cite{MR1901496}. But in this dimension, this material is certainly not original. For instance, between \cite[Chapter~3]{MR2067778} and \cite{hiptmair-2008}, almost all of the three-dimensional version of Theorem~\ref{alexander basis existence} has been published before. We do not yet know a reference for Theorem~\ref{alexander basis transformations}.

In what follows, we will take coefficients for homology to lie in $\R$. In this case\footnote{In dimension 3, \cite[p.119]{MR2067778} shows that since there is no torsion in any of the homology or cohomology groups, we could take coefficients in $\mathbb{Z}$ and get the same conclusions.}, the Universal Coefficient Theorem yields a natural duality isomorphism which pairs a homology class $x$ with the dual cohomology class $x^*$. We will let $[\Omega] \in H_{2k+1}(\Omega,\bdy\Omega)$ denote the top class of this orientable manifold and $[\Omega]^*$ denote the dual class in $H^{2k+1}(\Omega,\bdy\Omega)$. Similarly, $[\bdy\Omega] \in H_{2k}(\bdy\Omega)$ will be the top class of $\bdy\Omega$ and $[\bdy\Omega]^*$ its dual. We will let $\tilde{H}$ denote reduced homology. We will take the linking number of two $k$-cycles in $\R^{2k+1}$ to be given by $\Lk(a,b) = \Int(a,B)$ where $b = \bdy B$ and $\Int$ is the intersection number. We recall that for $k$-cycles, $\Lk(b,a) = (-1)^{(k+1)^2} \Lk(a,b)$~\cite[Proposition 11.13]{MR1454127}. 

We will also need a lemma.
\begin{lemma}
\label{whocares} 
For $n < 2k$, if we think of $\R^{2k+1}$ as $S^{2k+1} - \{x\}$ with $x \in \Int \ocomp$ then for a compact domain with boundary $\Omega \subset \R^{2k+1}$, 
\begin{equation*}
H_n(\R^{2k+1} - \Omega) = H_n(S^{2k+1} - \Omega).
\end{equation*}
\end{lemma}

\begin{proof}
Taking an open ball $D^{2k+1}$ around the point $x$, we have written $S^{2k+1}$ as a union of open sets. Then the (reduced) Mayer-Vietoris sequence yields an exact sequence
\begin{equation*}
\rightarrow H_n(S^{2k}) \rightarrow H_n(D^{2k+1}) \oplus H_n(\R^{2k+1} - \Omega) \rightarrow H_n(S^{2k+1} - \Omega) \rightarrow H_{n-1}(S^{2k}) \rightarrow
\end{equation*}
Since $D^{2k+1}$ is contractible (and we are in reduced homology), this provides the desired isomorphism immediately since the first and last homology groups in the sequence vanish. 
\end{proof}

We start with an existence theorem for a special basis for the $k$-th homology of $\bdy\Omega$:

\begin{theorem}
\label{alexander basis existence}
Let $\Omega$ be a compact domain with smooth boundary in $\R^{2k+1}$ or $S^{2k+1}$ (with $k > 0$) and $\ocomp$ be the complementary domain $\R^{2k+1} - \Omega$ or $S^{2k+1} - \Omega$. Then if we take coefficients in $\R$, $H_k(\bdy\Omega) = H_k(\Omega) \oplus H_k(\ocomp)$. Further, given any basis $\langle s_1, \dots, s_n \rangle$ for $H_k(\Omega)$ there is a corresponding basis $\langle s_1, \dots, s_n, t_1, \dots, t_n \rangle$ for $H_k(\bdy\Omega)$ which we call the \emph{Alexander basis} corresponding to $\langle s_1, \dots, s_n \rangle$ so that:
\begin{enumerate}
\item The inclusion $\bdy\Omega \hookrightarrow \Omega$ maps $\langle s_1, \dots, s_n \rangle \in H_k(\bdy\Omega)$ to the original basis $\langle s_1, \dots, s_n \rangle$ for $H_k(\Omega)$ and the inclusion $\bdy\Omega \hookrightarrow \ocomp$ maps $\langle t_1, \dots, t_n \rangle$ to a basis for $H_k(\ocomp)$.
\item 
\label{defs}
$s_i = \bdy \sigma_i$ for $\sigma_i \in H_{k+1}(\ocomp,\bdy\ocomp)$, where the $\sigma_i$ form a basis for $H_{k+1}(\ocomp,\bdy\ocomp)$. Similarly, $t_i = \bdy \tau_i$ for $\tau_i \in H_{k+1}(\Omega,\bdy\Omega)$, where the $\tau_i$ form a basis for $H_{k+1}(\Omega,\bdy\Omega)$.
\item 
\label{cup structure}
The cup product algebras of $\Omega$, $\ocomp$ and $\bdy\Omega$ obey 
\begin{equation*}
s_i^* \cup \tau_j^* = \delta_{ij} [\Omega]^*, \quad t_i^* \cup \sigma_j^* = (-1)^{k+1} \delta_{ij}[\ocomp]^*
\end{equation*} and
\begin{equation*}
s_i^* \cup s_j^* = 0, \quad t_i^* \cup s_j^* = \delta_{ij} [\bdy\Omega]^*, \quad t_i^* \cup t_j^* = 0.
\end{equation*}
\item 
\label{linking}
The linking number $\Lk(s_i,t_j) = \delta_{ij}$. (Thus $\Lk(t_j,s_i) = (-1)^{(k+1)^2} \delta_{ij}$.) 
\end{enumerate}
The Alexander duality isomorphism from $H_k(\Omega)$ to $H_k(\ocomp)$ maps $s_i$ to $t_i$.  
\end{theorem}

We will then study the effect of a homeomorphism on the Alexander basis, proving

\begin{theorem}
\label{alexander basis transformations}
Suppose that $\Omega$ and $\Omega'$ are compact domains with smooth boundary in $\R^{2k+1}$ or $S^{2k+1}$ and that $f\co \Omega \rightarrow \Omega'$ is an orientation-preserving homeomorphism. Then if $\langle s_1, \dots, s_n \rangle$ is a basis for $H_k(\Omega)$ and $\langle s_1', \dots, s_n' \rangle$ is a corresponding basis for $H_k(\Omega')$ so that $f_*(s_i) = s_i'$, then we may build Alexander bases $\langle s_1,\dots,s_n,t_1,\dots,t_n \rangle$ for $H_k(\bdy\Omega)$ and $\langle s_1',\dots,s_n',t_1',\dots,t_n' \rangle$ for $H_k(\bdy\Omega')$. For these bases, we have
$f_*(\tau_i) = \tau_i'$ and $\bdy f_*(t_i) = t_i'$ so that the map $\bdy f_*\co H_k(\bdy\Omega) \rightarrow H_k(\bdy\Omega')$ can be written as the $2n \times 2n$ matrix
\renewcommand{\arraystretch}{1.5}
\begin{equation}
\label{mmatrix}
M = \left[ 
	\begin{array}{c|c}
	I & 0 \\
	\hline 
	(c_{ij}) & I
	\end{array}
	\right],
\end{equation}
where each block represents an $(n \times n)$ matrix. If $k$ is odd, the block matrix $c_{ij}$ is symmetric, while if $k$ is even, the block matrix $c_{ij}$ is skew-symmetric.
\end{theorem}

An example of these theorems was shown in Figure~\ref{fig:alexanderbasis}. 
We now restate some of our main tools for this theorem. The first is a form of Poincar\'{e} duality~\cite[Theorem 3.4.3, p.254]{hatcher}:

\begin{theorem}[Lefschetz Duality]
\label{lefdual}
Suppose $M$ is a compact orientable $n$-manifold with boundary $\bdy M$. Then cap product with a top class $[M] \in H^n(M, \bdy M)$ gives isomorphisms $D_M \co H^i(M) \rightarrow H_{n-i}(M, \bdy M)$ for all $k$. 
\end{theorem}

We now begin the proof. We are essentially reproving the Alexander duality theorem while recording additional information along the way.
\begin{proof}[Proof of Theorem~\ref{alexander basis existence}]
Let us restrict our attention to the case where $\Omega \subset S^{2k+1}$ for convenience. (Lemma~\ref{whocares} shows that the same proof works in both cases.)
We observe that since $\bdy\Omega$ is compact and smooth, it has a open tubular neighborhood which deformation retracts to $\bdy\Omega$. Using this, the reduced Mayer-Vietoris sequence \cite[p.150]{hatcher} yields a long exact sequence including
\begin{equation*}
\label{mvseq}
H_{k+1}(S^{2k+1}) = 0 \rightarrow H_k(\bdy\Omega) \xrightarrow{\Phi} H_k(\Omega) \oplus H_k(\ocomp) \xrightarrow{\Psi} H_{k}(S^{2k+1}) = 0.
\end{equation*}
This proves that the map $\Phi$ given by the inclusions of $\bdy\Omega$ into $\Omega$ and $\ocomp$ is an isomorphism between $H_k(\bdy\Omega)$ and $H_k(\Omega) \oplus H_k(\ocomp)$, proving the first statement in our theorem\footnote{In fact, a version of this observation actually predates homology theory itself! The statement that (essentially) the dimension of $H_1(\bdy\Omega)$ was equal to the sum of the dimensions of $H_1(\Omega)$ and $H_1(\ocomp)$ was published by James Clerk Maxwell in 1891 in his \emph{Treatise on Electricity and Magnetism}~\cite{maxwell}. Steenrod gives the result in a form essentially the same as ours as a theorem of Hopf~\cite[Theorem 2.1]{MR0145525}, while Kauffman proves that for any three manifold with boundary, half of the homology of the boundary is in the kernel of the inclusion of the boundary into the interior~\cite[Lemma 8.1]{MR907872}.}

Using this isomorphism, we see that our original basis $s_1, \dots, s_n \in H_k(\Omega)$ is the $\Phi$-image of a linearly independent set of classes $s_1, \dots, s_n \in H_k(\bdy\Omega)$ which vanish when included in $H_k(\ocomp)$. By Lefschetz Duality (\ref{lefdual}), 
\begin{equation*}
[\Omega] \cap \co H^k(\Omega) \rightarrow H_{k+1}(\Omega,\bdy\Omega), 
\end{equation*}
is an isomorphism. Let $\tau_1, \dots, \tau_n$ denote the images of the cohomology duals of $s_1, \dots, s_n$ under this isomorphism. By construction, $[\Omega] \cap s_i^* = \tau_i$, or $\tau_j^* ( [\Omega] \cap s_i^*) = \delta_{ij}$. Now for $\alpha \in H_{k+l}(X,\bdy X)$, $\phi \in H^k(X)$ and $\psi \in H^{l}(X,\bdy X)$, the cup and cap product are related by the formula~\cite[p.249]{hatcher},
\begin{equation}
\label{capandcup}
\psi( \alpha \cap \phi ) = (\phi \cup \psi) (\alpha).
\end{equation}
Taking $\psi = \tau_j^*$, $\alpha = [\Omega]$, and $\phi = s_i^*$, we get 
\begin{equation*}
1 = \tau_j^* ( [\Omega] \cap s_i^*) = (s_i^* \cup \tau_j^*) ([\Omega]).
\end{equation*}
This proves our statement about the cup structure of $\Omega$.

We now map the $\tau_i$ to classes in $H_k(\ocomp)$. If we take a small open neighborhood $\bdy\Omega_o$ of $\bdy\Omega$ which deformation retracts to $\bdy\Omega$, we can construct $\Omega_o = \Omega \cup \bdy\Omega_o$ so that the closure of $\ocomp_o = S^{2k+1} - \Omega_o$ is contained in the interior of $\ocomp$. Clearly $\Omega_o = S^{2k+1} - \ocomp_o$ and 
\begin{equation*}
\ocomp - \ocomp_o = (S^{2k+1} - \Omega) - (S^{2k+1} - (\Omega \cup \bdy\Omega_o)) = \bdy\Omega_o.
\end{equation*}
This means that the pair $(\Omega_o,\bdy\Omega_o) = (S^{2k+1} - \ocomp_o,\ocomp - \ocomp_o)$. Further, since the closure of $\ocomp_o$ is contained in the interior of $\ocomp$, the inclusion of $(S^{2k+1} - \ocomp_o,\ocomp - \ocomp_o) \hookrightarrow (S^{2k+1},\ocomp)$ induces the homology isomorphism \begin{equation*}
H_{k+1}(\Omega,\bdy\Omega) = H_{k+1}(\Omega_o,\bdy\Omega_o) =  H_{k+1}(S^{2k+1},\ocomp).
\end{equation*}
by deformation retraction of $(\Omega_o,\bdy\Omega_o)$ onto $(\Omega,\bdy\Omega)$ and excision~\cite[Theorem 2.20, p.119]{hatcher}. But the exact sequence of the pair $(S^{2k+1},\ocomp)$ contains
\begin{equation*}
H_{k+1}(S^{2k+1}) = 0 \rightarrow H_{k+1}(S^{2k+1},\ocomp) \xrightarrow{\bdy} H_k(\ocomp) \rightarrow H_k(S^{2k+1}) = 0
\end{equation*}
So the boundary map carries the $\tau_i$ to a set of generators $t_i$ for $H_k(\ocomp)$. The entire map we have built from $s_i$ (as a basis for $H_k(\Omega)$) to $t_i$ (as a basis for $H_k(\ocomp)$) is the Alexander duality isomorphism.

Since we can pull these $t_i$ back to $H_k(\bdy\Omega)$ under the isomorphism $\Phi\co H_k(\bdy\Omega) \rightarrow H_k(\Omega) \oplus H_k(\ocomp)$, we can regard the $t_i$ as a linearly independent set of elements in $H_k(\bdy\Omega)$ which complete the Alexander basis $\langle s_1, \dots, s_n, t_1, \dots, t_n \rangle$ for $H_k(\bdy\Omega)$. In fact, we can choose representatives for the $t_i$ so that $t_i = \bdy\tau_i$.  We now observe that $\Lk(s_i,t_j) = \Int(s_i,\tau_j) = (s_i^* \cup \tau_j^*)[\Omega] = \delta_{ij}$.

We now work out the cup product of $s_i^*$ and $t_j^*$. The relation between the cap product and the boundary operator for an $i$-chain $\alpha$ in $C_i(X,A)$ and a cochain $\beta \in C^l(X)$ is given by
\begin{equation*}
\label{capandbdy}
\bdy (\alpha \cap \beta) = (-1)^l (\bdy \alpha \cap \beta - \alpha \cap \delta \beta)
\end{equation*}
where $\delta$ is the coboundary operator \cite[p.240]{hatcher}.  We now compute
\begin{equation}
\label{ocapsi}
t_i = \bdy \tau_i = \bdy ( [\Omega] \cap s_i^* ) = (-1)^k (\bdy [\Omega] \cap s_i^* - [\Omega] \cap \delta s_i^*) = (-1)^k ([\bdy\Omega] \cap s_i^*),
\end{equation}
where $\delta s_i^* = 0$ because $s_i^*$ is a cocycle. We can compute
\begin{equation}
\label{sicupti}
\delta_{ij} = t_j^*(t_i) = (-1)^k t_j^* ([\bdy \Omega] \cap s_i^*) = (-1)^k (s_i^* \cup t_j^*) [\bdy\Omega].
\end{equation}
Now we recall that the cup product of $\alpha \in H^i(X)$ and $\beta \in H^j(X)$ obeys the (anti)commutativity relation \cite[Theorem 3.14, p.215]{hatcher}:
\begin{equation*}
\label{cupcommutes}
\alpha \cup \beta = (-1)^{ij} \beta \cup \alpha.
\end{equation*}
Using this equation, we see immediately that~\eqref{sicupti} yields $(t_j^* \cup s_i^*)[\bdy\Omega] = \delta_{ij}$ whether $k$ is even or odd. Thus we have
\begin{equation*}
\delta_{ij} = (t_j^* \cup s_i^*)[\bdy \Omega] = s_i^*([\bdy\Omega] \cap t_j^*),
\end{equation*}
and $[\bdy\Omega] \cap t_i^* = s_i + d_l t_l$. (We will shortly argue that the $d_l$ are all zero.)

As above, we know $[\ocomp] \cap \co H^k(\ocomp) \rightarrow H_{k+1}(\ocomp,\bdy\ocomp)$ is an isomorphism. We let $\sigma_i = (-1)^{k+1} [\ocomp] \cap t_i^*$. We now show $\bdy\sigma_i \in H_k(\bdy\Omega) = s_i$. We notice first that $\bdy\sigma_i$ is a linear combination of $s_j$ since it is certainly the case that $\bdy\sigma_i$ bounds in $\ocomp$, so $\bdy\sigma_i$ must be contained in the $H_k(\Omega)$ summand of $H_k(\bdy\Omega) = H_k(\Omega) \oplus H_k(\ocomp)$.

Now we compute
\begin{equation}
\label{ocapti}
\begin{split}
\bdy\sigma_i &= \bdy ((-1)^{k+1} [\ocomp] \cap t_i^*) = (-1)^{2k+1} (\bdy[\ocomp] \cap t_i^* - [\ocomp] \cap \delta t_i^*) \\
&= (-1)^{2k+1}(-[\bdy\Omega] \cap t_i^*) = [\bdy\Omega] \cap t_i^*.
\end{split}
\end{equation}
We already know that $[\bdy\Omega] \cap t_i^* = s_i + d_l t_l$. Since $\bdy\sigma_i$ is a linear combination of $s_j$, we have shown that $\bdy\sigma_i = s_i$, as desired.

Notice that we needed the $(-1)^{k+1}$ in our definition of $\sigma_i = (-1)^{k+1}[\ocomp] \cap t_i^*$ to make the signs come out correctly in this last sequence of arguments. To verify this sign, observe that we can directly compute
\begin{align*}
\Lk(t_j,s_i) &= \Int(t_j,\sigma_i) = (t_j^* \cup \sigma_i^*)([\ocomp]) \\ 
 &= \sigma_i^* ([\ocomp] \cap t_j^*) = (-1)^{k+1} \sigma_i^*(\sigma_j) = (-1)^{k+1} \delta_{ij},
\end{align*}
which agrees with our previous computation 
\begin{equation*}
\Lk(t_j,s_i) = (-1)^{(k+1)^2} \Lk(s_i,t_j) = (-1)^{(k+1)^2} \delta_{ij}.
\end{equation*}
Now consider the cup product $s_i^* \cup s_j^*$. Using \eqref{capandcup} and \eqref{ocapsi}, we have
\begin{equation*}
(s_i^* \cup s_j^*)[\bdy\Omega] = s_j^* ([\bdy\Omega] \cap s_i^*) = (-1)^k s_j^* (t_i) = 0.
\end{equation*}
Similarly, using~\eqref{capandcup} and~\eqref{ocapti}, we have
\begin{equation*}
(t_i^* \cup t_j^*)[\bdy\Omega] = t_j^* ([\bdy\Omega] \cap t_i^*) = t_j^* (s_i) = 0.  \qedhere
\end{equation*}
\end{proof}

We now prove our second theorem about the Alexander basis.

\newcommand{\sip}{s_i'}
\newcommand{\tjp}{t_j'}
\newcommand{\sips}{s_i'^*}
\newcommand{\tjps}{t_j'^*}
\newcommand{\taujps}{(\tau_j')^*}

\begin{proof}[Proof of Theorem~\ref{alexander basis transformations}]
Since $f$ is an orientation-preserving homeomorphism, we have $f_*[\bdy\Omega] = [\bdy\Omega']$. We also know that in $H_k(\Omega')$ and $H_k(\Omega)$, $f^*(\sips) = s_i^*$ since we have $f_*(s_i) = \sip$.  This means that in $H_k(\bdy\Omega)$, $\bdy f_*(s_i) = \sip + c_{ji} t_j'$ for some coefficients $c_{ji}$, since the $t_j'$ span the kernel of the inclusion homomorphism from $H_k(\bdy\Omega')$ to $H_k(\Omega')$.

We now compute $f^*(\tau_j'^*) \in H^{k+1}(\Omega,\bdy\Omega)$. We know
\begin{equation*}
f^*(\sips \cup \taujps) = f^*(\sips) \cup f^*(\taujps) = s_i^* \cup f^*(\taujps).
\end{equation*}
On the other hand by conclusion~(\ref{cup structure}) of 
Theorem~\ref{alexander basis existence}, we know
\begin{equation*}
f^*(\sips \cup \taujps) = f^*(\delta_{ij} [\Omega']^*) = \delta_{ij} [\Omega]^*.
\end{equation*}
Thus, using our construction of $\tau_i = [\Omega] \cap s_i$, 
\begin{equation*}
\delta_{ij} = (s_i^* \cup f^*(\taujps))[\Omega] = f^*(\taujps) ([\Omega] \cap s_i^*) = f^*(\taujps)(\tau_i).
\end{equation*}
and we may conclude that $f^*(\taujps) = \tau_j^*$. Hence $f_*(\tau_j) = \tau_j'$, and $f_*(t_i) = t_i'$, since $t_i = \bdy\tau_i$. 

This proves that the map $\bdy f_*\co H_k(\bdy\Omega) \rightarrow H_k(\bdy\Omega')$ can be written in the matrix form above. Since $\bdy f_*(s_i) = s_i' + c_{ji} t_j'$, it follows that $s_i^* = \bdy f^*(s_i'^* + c_{ji} t_j'^*)$. Of course, using conclusion~(\ref{cup structure}) of Theorem~\ref{alexander basis existence} again, we know 
\begin{equation*}
0 = s_i^* \cup s_j^* = \bdy f^*(s_i'^* + c_{ki} t_k'^*) \cup \bdy f^*(s_j'^* + c_{jl} t_l'^*).
\end{equation*}
In addition, we know $\bdy f^*$ is an isomorphism and $t_i^* \cup s_j^* = \delta_{ij} [\bdy\Omega']^*$ so
\begin{align*}
0 &= (\sips + c_{ik} t_k'^*) \cup (s_j'^* + c_{jl} t_l'^*) = c_{jl} s_i'^* \cup t_l'^* + c_{ik} t_k'^* \cup s_j'^* \\
  &= (c_{jl} (-1)^{k^2} \delta_{il} + c_{ik} \delta_{kj}) [\bdy\Omega']^* = ((-1)^{k^2} c_{ji} + c_{ij}) [\bdy\Omega']^*.
\end{align*}
If $k$ is even, this equation becomes $c_{ji} + c_{ij} = 0$ and the matrix is skew-symmetric, while if $k$ is odd, this equation becomes $-c_{ji} + c_{ij} = 0$ and the matrix is symmetric, as claimed.
\end{proof}

It is tempting to wonder what happens if one computes the cup products $s_i^* \cup t_j^* = \bdy f^*(s_i'^* + c_{ki} t_k'^*) \cup \bdy f^*(t_j')$ or $t_i^* \cup t_j^* = \bdy f^*(t_i'^*) \cup \bdy f^*(t_j'^*)$. It is certainly possible to do so, yielding expansions similar to those above, but it turns out to be the case that this procedure yields no additional information about the $c_{ij}$. Thus we believe that Theorems~\ref{alexander basis existence} and~\ref{alexander basis transformations} summarize all of the cohomological information available for an arbitrary compact domain with boundary in $\R^{2k+1}$.

\section{The Hodge-Morrey-Friedrichs Decomposition for Manifolds with Boundary} 
\label{hodge}  

The purpose of this section is to explain why every closed Dirichlet form on a domain $\Omega$ in $\R^{2k+1}$ is exact using a decomposition of the differential forms on $\Omega$, which for now we view as a smooth, compact, Riemannian $n$-manifold with non-empty boundary.  As with the previous appendices, this section is expository; see Schwarz \cite{MR1367287} for full details.  Non-experts may prefer the well-written treatment of this subject in Chapter 2 of Clayton Shonkwiler's thesis~\cite{shonk}.  

\newcommand{\exact}[1]{\mathcal{E}^{#1}(\Omega)}
\newcommand{\coexact}[1]{c\mathcal{E}^{#1}(\Omega)}
\newcommand{\harmonic}[1]{\mathcal{H}^{#1}(\Omega)}
\newcommand{\harmonicf}[1]{\widehat{\mathcal{H}}^{#1}(\Omega)}

\newcommand{\exactharmonic}[1]{\mathcal{EH}^{#1}(\Omega)}
\newcommand{\coexactharmonic}[1]{c\mathcal{EH}^{#1}(\Omega)}

\newcommand{\forms}[1]{\Lambda^{#1}(\Omega)}

\newcommand{\nexact}[1]{\mathcal{E}^{#1}_N(\Omega)}
\newcommand{\ncoexact}[1]{c\mathcal{E}^{#1}_N(\Omega)}
\newcommand{\nharmonic}[1]{\mathcal{H}^{#1}_N(\Omega)}

\newcommand{\dexact}[1]{\mathcal{E}^{#1}_D(\Omega)}
\newcommand{\dcoexact}[1]{c\mathcal{E}^{#1}_D(\Omega)}
\newcommand{\dharmonic}[1]{\mathcal{H}^{#1}_D(\Omega)}

\newcommand{\exactcoexact}[1]{\mathcal{E}c\mathcal{E}^{#1}(\Omega)}

We let $\forms{p}$ denote the vector space of smooth $p$-forms on $\Omega$. The exterior derivative $d$ and the codifferential $\delta$ map to $\forms{p+1}$ and $\forms{p-1}$, respectively.  The Laplacian on $p$-forms is defined as $\Delta = d\delta + \delta d$.  We can define natural subspaces of $\forms{p}$ by looking at the kernel and image of $d$ and $\delta$:

\begin{definition}
The kernel of $d$ is the space of \emph{closed} forms. The kernel of $\delta$ is the space of \emph{co-closed} forms. The intersection of these spaces is the space of \emph{harmonic $p$-fields} $\harmonic{p}$.  These form a subset of the \emph{harmonic $p$-forms} $\harmonicf{p}$, defined as the kernel of $\Delta$.

We use the following notation for the image of $d$ and $\delta$:
\begin{itemize}
\item 
\emph{exact} $p$-forms: $\exact{p} \subset \forms{p}$ is the image of $d \co \forms{p-1} \rightarrow \forms{p}$.
\item 
\emph{co-exact} $p$-forms: $\coexact{p} \subset \forms{p}$ is the image of $\delta \co \forms{p+1} \rightarrow \forms{p}$.
\end{itemize}

We can take the intersections of appropriate subspaces to yield the exact harmonic $p$-fields $\exactharmonic{p}$ and the co-exact harmonic $p$-fields $\coexactharmonic{p}$.
\end{definition}

Since the boundary of $\Omega$ is nonempty, we can also classify forms into subspaces by their behavior on the boundary. 

\begin{definition}
A $p$-form $\omega \in \forms{p}$ obeys \emph{Dirichlet boundary conditions} if $\omega$ vanishes when restricted to $\bdy\Omega$, i.e., if $V_1, \ldots, V_p$ all lie in $T_x{\bdy\Omega}$, then $\omega(V_1, \ldots, V_p)=0$.  We say $\omega$ obeys \emph{Neumann boundary conditions} when the Hodge dual $\star \omega$ of $\omega$ obeys Dirichlet boundary conditions.
\end{definition}

We apply boundary conditions to the exact and co-exact forms and harmonic fields  by attaching a subscript $D$ or $N$ to the notation above. For the harmonic fields, $\dharmonic{p}$ and $\nharmonic{p}$ are just the intersections of harmonic fields with Dirichlet and Neumann forms. For the exact and co-exact forms, the boundary condition applies also to the \emph{primitives} of the forms in $\nexact{p}$, $\ncoexact{p}$, $\dexact{p}$, and $\dcoexact{p}$. 

We can now give

\begin{theorem}[Hodge-Morrey-Friedrichs Decomposition Theorem~\cite{MR1367287}]
\label{hmf}
Let $\Omega$ be a compact, oriented, smooth Riemannian manifold with non-empty boundary $\bdy \Omega$. Then $\forms{p}$ admits the $L^2$ orthogonal decompositions 
\begin{align}
\forms{p} &= \ncoexact{p} \oplus \; \nharmonic{p} \; \oplus \exactharmonic{p} \oplus \dexact{p} \label{hmf1}\\
\forms{p} &= \ncoexact{p} \oplus \coexactharmonic{p} \oplus \dharmonic{p} \, \oplus \dexact{p}. \label{hmf2}
\end{align}
Further, $\nharmonic{p} \simeq H^p(\Omega;\R)$ while $\dharmonic{p} \simeq H^p(\Omega,\bdy\Omega; \R)$. 
\end{theorem} 

Using the first decomposition, we now make an observation about closed Dirichlet forms on $\Omega$.

\begin{proposition}
\label{alphaisexact}
Let $0<p < 2k+1$.  The space of closed Dirichlet $p$-forms is $\dharmonic{p} \oplus \dexact{p}$.  If $\Omega$ is a compact subdomain of $\R^{2k+1}$ with smooth boundary, then every closed Dirichlet form $\alpha \in \forms{p}$ is exact.
\end{proposition}

Our notation has the unfortunate consequence of making it easy to misread this proposition.  For Euclidean subdomains, the proposition does imply that every form in $\dharmonic{p}$ must be exact. It is tempting to conclude that this means that $\dharmonic{p} \subset \dexact{p}$, contradicting the orthogonality of the Hodge-Morrey-Friedrichs decomposition. This is not the case.  For a form $\alpha$ to be in $\dexact{p}$ it is necessary but \emph{not} sufficient that $\alpha$ be exact and Dirichlet: $\alpha$ must also have a Dirichlet primitive.  The exact forms that are Dirichlet but fail to have a Dirichlet primitive all lie in $\dharmonic{p}$.

\begin{proof}
We begin by showing that the subspace $\ncoexact{p}$ is the orthogonal complement of the closed forms on $\Omega$.  
Any form $\alpha \in \ncoexact{p}$ is co-exact, hence co-closed; if $\alpha$ were also closed, it would be a harmonic field lying in $\nharmonic{p}$ which has trivial intersection with $\ncoexact{p}$ by \eqref{hmf1}.  Since harmonic fields and exact forms are both closed, all of the other three summands in \eqref{hmf1} are closed, which shows that $\forms{p} = \ncoexact{p} \oplus \{\text{closed\ } p\text{-forms}\}$.

We now can apply the second decomposition \eqref{hmf2} to say that the closed Dirichlet forms must lie inside $\coexactharmonic{p} \oplus \dharmonic{p} \oplus \dexact{p}$.  No $p$-form in $\coexactharmonic{p}$ can be Dirichlet, since it is also a harmonic field and the harmonic Dirichlet fields comprise $\dharmonic{p}$, which lies orthogonal to $\coexactharmonic{p}$.  Thus we conclude that the subspace of closed Dirichlet $p$-forms is precisely $\dharmonic{p} \oplus \dexact{p}$; all of these forms are both closed and Dirichlet. 

Now we turn toward the second statement of the proposition.  Let $\Omega$ now be a compact subdomain of $\R^{2k+1}$ with smooth boundary.  Using \eqref{hmf1}, we can decompose any closed Dirichlet form $\alpha \in \forms{p}$ as
\begin{equation*}
\alpha = \alpha_{\mathcal{H}_N} + \alpha_{\mathcal{EH}} + \alpha_{\mathcal{E}_D},
\end{equation*}
since it has no component in $\ncoexact{p}$.  In other words, $\alpha$ is the sum of a form in $\nharmonic{p}$ and an exact form. Further, since $\nharmonic{p} \simeq H^{p}(\Omega;\R)$, we can determine $\alpha_{\mathcal{H}_N}$ by finding the cohomology class represented by $\alpha$ in $H^{p}(\Omega)$. 

We claim that this cohomology class is zero since $\alpha$ is Dirichlet.  To see this, recall that if $\ocomp$ is the complement of $\Omega$ in $\R^{2k+1}$, the Mayer-Vietoris sequence for $\R^{2k+1} = \Omega \cup \ocomp$ includes
\begin{equation*}
0 = H_{p-1}(\R^{2k+1}) \rightarrow H_{p}(\bdy\Omega) \xrightarrow{i_* \oplus i'_*} H_{p}(\Omega) \oplus H_{p}(\ocomp) \rightarrow H_{p}(\R^{2k+1})  = 0
\end{equation*}
where $i$ is the inclusion $\bdy\Omega \hookrightarrow \Omega$.  Thus in particular  $i_* \co H_{p}(\bdy\Omega) \rightarrow H_{p}(\Omega)$ is onto.  Now this means that for any cycle $[X] \in H_{p}(\Omega)$, we can find $Y \in \bdy\Omega$ so that $i_*([Y]) = [X]$.  Then, since a Dirichlet form vanishes on the boundary,
\begin{equation*}
\int_X \alpha = \int_Y i^*(\alpha) = \int_Y 0 = 0.
\end{equation*}
This means that $\alpha_{\mathcal{H}_N} = 0$, and so $\alpha$ is exact. 
\end{proof}

One more subspace appears in our next decomposition.  Define $\exactcoexact{p}$ to be all $p$-forms that are both exact and co-exact; all such forms are clearly harmonic $p$-fields.
\begin{corollary}[cf., {\cite[Theorem 2.1.1]{shonk}}] For $\Omega$ is a compact subdomain of $\R^{2k+1}$ with smooth boundary, the following decomposition holds
\begin{equation} \label{hmfdg}
\forms{p} = \ncoexact{p} \oplus \exactcoexact{p} \oplus \nharmonic{p} \oplus \dharmonic{p} \oplus \dexact{p}.
\end{equation}
\end{corollary}

\begin{proof}
The proposition above shows all forms in $\dharmonic{p}$ are exact, hence $\dharmonic{p} \subset \exactharmonic{p}$, which lies orthogonal to $\nharmonic{p}$.  The orthogonal complement of $\dharmonic{p}$ within $\exactharmonic{p}$ are those harmonic fields that are both exact and co-exact.  Hence, the harmonic $p$-fields decompose as
\begin{equation*} 
\harmonic{p} = \exactcoexact{p} \oplus \nharmonic{p} \oplus \dharmonic{p},
\end{equation*}
which along with the Hodge-Morrey-Friedrichs decomposition proves this corollary.
\end{proof}

We conclude with one more decomposition, for vector fields on a domain in $\R^3$; see \cite{MR1901496} for full details.  

\begin{theorem}[Hodge Decomposition Theorem for vector fields in $\R^3$~\cite{MR1901496}] \label{thm:hodgevf}
Let $\Omega$ be a compact subdomain of $\R^3$ with smooth $\partial\Omega$. Then, the space of smooth vector fields $VF(\Omega)$ decomposes into five mutually orthogonal subspaces,
\begin{equation} \label{hodgevf}
VF(\Omega)=FK \oplus HK \oplus CG \oplus HG \oplus GG,
\end{equation}
where,
\begin{center}
\begin{displaymath}
\begin{array}{cclcl}
FK & = & \mathrm{fluxless\ knots} & = & \left\{\diver{V}=0, \; V \cdot n =0, \; \mathrm{all\ interior\ fluxes} = 0\right\}\\
HK & = & \mathrm{harmonic\ knots} & = & \{\diver{V}=0, \; V \cdot n =0, \; \curl{V}=0\} \\
CG  & = & \mathrm{curly\ gradients}  & = & \{V=\nabla\phi, \; \diver{V}=0, \; \mathrm{all\ boundary\ fluxes} = 0\} \\
HG & = & \mathrm{harmonic\ gradients} & = & \{V=\nabla\phi, \; \diver{V}=0, \; \phi \mathrm{\ locally\ const.\ on\ } \partial\Omega\} \\
GG & = & \mathrm{grounded\ gradients} & = & \{V=\nabla\phi, \; \phi|_{\partial\Omega}=0\}
\end{array}
\end{displaymath}
\end{center}
\end{theorem}

\vspace{2ex}
\begin{corollary}
The Hodge Decomposition Theorem for vector fields precisely corresponds to the five-term decomposition \eqref{hmfdg} for 2-forms on $\Omega$.  The five summands in each decomposition are isomorphic as vector spaces, and pair as follows:
\begin{equation} \label{five-terms}
c\mathcal{E}_N \cong GG, \quad \mathcal{E}c\mathcal{E} \cong CG, \quad \mathcal{H}_N \cong HG, \quad \mathcal{H}_D  \cong HK, \quad \mathcal{E}_D \cong FK . 
\end{equation}
\end{corollary}

\begin{proof}
To prove this corollary, we consider a 2-form $\alpha$ on a subdomain $\Omega \subset \R^3$ and its dual vector field $V$.  We translate our definitions regarding forms into the statements about vector fields.

\begin{itemize}
\item $\alpha$ is {\it closed}, i.e., $d\alpha = 0$ $\Longleftrightarrow$ $V$ is divergence-free.
\item $\alpha$ is {\it exact}, i.e., $\alpha = d\beta$ for a 1-form $\beta$  $\Longleftrightarrow$ $V$ lies in the image of curl.
\item $\alpha$ is {\it co-closed}, i.e., $\delta\alpha = \star d \star \alpha = 0$.  Here $d \star \alpha$ amounts to taking the curl of $V$, so co-closed corresponds to $V$ lying in the kernel of curl.
\item $\alpha$ is {\it co-exact}, i.e., $\alpha = \delta\gamma = \star d \star \gamma $, for some 3-form $\gamma$.  Any 3-form can be written as $\gamma = f \dVol$, so $\star \gamma = f$.  Thus, co-exact corresponds to $V$ equaling a gradient $\nabla f$.
\item $\alpha$ is {\it Dirichlet} $\Longleftrightarrow$ $V$ is tangent to the boundary.
\item $\alpha$ is {\it Neumann}, meaning $\star \alpha$ is Dirichlet $\Longleftrightarrow$ $V$ is normal to the boundary.
\item $\alpha \in \ncoexact{2} \Longleftrightarrow V = \nabla f$ and $f=0$ on the boundary
\end{itemize}

Now, let us examine each piece of \eqref{five-terms}.  The first piece, $\ncoexact{2}$, corresponds to gradients that vanish on $\bdy\Omega$; this defines the subspace $GG$.  The exact and co-exact forms correspond to gradients that lie in the image of curl, namely $CG$.  

The third piece, $\nharmonic{2}$ are closed and co-closed Neumann forms.  (Combining \eqref{hmf2} and \eqref{hmfdg}, we see that forms in $\nharmonic{p}$ actually co-exact.)  The forms in $\nharmonic{2}$ are the only closed forms that are not exact.  These correspond to vector fields that are divergence-free but not in the image of curl comprise $HG$.  

The fourth piece, $\dharmonic{2}$ are exact, co-closed Dirichlet forms.  These correspond to divergence-free fields which are tangent to the boundary and lie in the kernel of curl, which characterizes $HK$.  As an additional check, $\nharmonic{2} \simeq H^2(\Omega; \R) \simeq HG$ and $\dharmonic{2} \simeq H^1(\Omega; \R) \simeq HK$.

The final piece $\dexact{2}$ consists of exact forms whose primitives obey a Dirichlet condition.  The corresponding vector field then must be divergence-free and tangent to the boundary; furthermore, the Dirichlet condition on the primitive implies that all fluxes of the vector field over surfaces $\Sigma \subset \Omega$ with $\bdy\Sigma \subset \bdy\Omega$ must vanish.  This precisely characterizes $FK$.
\end{proof}

\end{document}